\documentclass[onefignum,onetabnum]{siamart171218}

\usepackage{amsfonts}
\usepackage{amssymb}
\usepackage{graphicx}
\usepackage{epstopdf}
\usepackage{algorithmic}
\ifpdf
  \DeclareGraphicsExtensions{.eps,.pdf,.png,.jpg}
\else
  \DeclareGraphicsExtensions{.eps}
\fi

\usepackage{array}
\usepackage{bm}
\usepackage{tikz-cd}
\usepackage{subfig}
\usepackage{mathtools}
\newsiamremark{remark}{Remark}
\newsiamremark{assumption}{Assumption}

\newsiamremark{hypothesis}{Hypothesis}
\crefname{hypothesis}{Hypothesis}{Hypotheses}
\newsiamthm{claim}{Claim}

\def\Xint#1{\mathchoice
  {\XXint\displaystyle\textstyle{#1}}%
  {\XXint\textstyle\scriptstyle{#1}}%
  {\XXint\scriptstyle\scriptscriptstyle{#1}}%
  {\XXint\scriptscriptstyle\scriptscriptstyle{#1}}%
\!\int}
\def\XXint#1#2#3{{\setbox0=\hbox{$#1{#2#3}{\int}$ }
\vcenter{\hbox{$#2#3$ }}\kern-.6\wd0}}

\def\dashint{\Xint-}

\newcommand{\grad}{{\rm grad}}
\newcommand{\curl}{{\rm curl}}

\newcommand{\tr}{{\rm tr}}
\renewcommand{\div}{{\rm div}}

\newcommand\commentone[1]{\textcolor{black}{#1}}
\headers{Simplex-Averaged Finite Element Methods}{S. Wu and J. Xu}

\title{Simplex-averaged finite element methods for $H(\grad)$,
$H(\curl)$ and $H(\div)$ convection-diffusion
problems\thanks{%Submitted to the editors DATE.  
The work of Shuonan Wu is supported in part by the startup grant from
Peking University.  The work of Jinchao Xu is supported in part by the
US Department of Energy Office of Science, Office of Advanced
Scientific Computing Research, Applied Mathematics program under Award
Number DE-SC0014400.}
}

\author{Shuonan Wu\thanks{School of Mathematical Sciences,
  Peking University, Beijing 100871, China 
(\email{snwu@math.pku.edu.cn}, \url{http://dsec.pku.edu.cn/\~snwu}).}
\and Jinchao Xu\thanks{Department of Mathematics, Pennsylvania State
University, University Park, PA 16802, USA
(\email{xu@math.psu.edu}, \url{http://www.math.psu.edu/xu/}).}
}

\usepackage{amsopn}

\ifpdf
\hypersetup{
  pdftitle={Simplex-Averaged Finite Element Methods},
  pdfauthor={S. Wu, and J. Xu}
}
\fi

\begin{document}

\maketitle

% REQUIRED
\begin{abstract}
This paper is devoted to the construction and analysis of the finite
element approximations for the $H(D)$ convection-diffusion problems,
where $D$ can be chosen as ${\rm grad}$, ${\rm curl}$ or ${\rm div}$
in 3D case. An essential feature of these constructions is to properly
average the PDE coefficients on the sub-simplexes. The schemes are of
the class of exponential fitting methods that result in special
upwind schemes when the diffusion coefficient approaches to zero.
Their well-posedness are established for sufficiently small mesh size
assuming that the convection-diffusion problems are uniquely solvable.
Convergence of first order is derived under minimal smoothness of the
solution. Some numerical examples are given to demonstrate the
robustness and effectiveness for general convection-diffusion
problems.
\end{abstract}

% REQUIRED
\begin{keywords}
Convection-diffusion problems, finite element methods, discrete
differential forms, exponential fitting, magnetohydrodynamics 
\end{keywords}

% REQUIRED
\begin{AMS}
65N30, 65N12, 65N15
\end{AMS}

%%%% Introduction %%%% 
\section{Introduction} \label{sec:intro}
%% General introduction: convection-diffusion
The $H(\grad)$, $H(\curl)$ and $H(\div)$ convection-diffusion
problems, especially the convection dominated ones, arise in many
important applications.  To fix ideas, we consider a simple example
taken from magnetohydrodynamics \cite{gerbeau2006mathematical},
$$ 
\left\{
\begin{aligned}
& \bm{j} - R_m^{-1} \nabla \times (\mu_r^{-1}\bm{B}) = 0, \\
& \bm{B}_t + \nabla \times \bm{E} = 0, \\
& \bm{j} = \sigma_r(\bm{E} + \bm{v}\times \bm{B}), \\
& \nabla \cdot \bm{B} = 0.
\end{aligned}
\right.
$$ 
Physically, $\bm{E}$ and $\bm{B}$ are the non-dimensionalized electric
field and magnetic field inside a conductor moving with a velocity
$\bm{v}$, respectively. The physical parameters are the magnetic
Reynolds number $R_m$, the relative electrical conductivity
$\sigma_r$, and the relative magnetic permeability $\mu_r$. With a
simple implicit time-discretization on the Faraday's Law and
eliminations of the magnetic field $\bm{E}$ and the current density
$\bm{j}$, the electric field satisfies the following $H(\curl)$
convection-diffusion equation: 
\begin{equation} \label{eq:E-cd}
\nabla \times (\alpha \nabla \times \bm{E}) - \bm{\beta} \times
(\nabla \times \bm{E}) + \gamma \bm{E} = \bm{f}, 
\end{equation} 
where $\alpha = R_m^{-1}\mu_r^{-1}$, $\bm{\beta} = \sigma_r \bm{v}$,
$\gamma = \sigma_r / k$ and 
$$ 
\bm{f} = \frac{1}{k}\nabla \times (R_m^{-1}\mu_r^{-1}\bm{B}^{-}) -
\frac{1}{k}\sigma_r \bm{v}\times \bm{B}^{-},
$$
and $\bm{B}^{-}$ is a known magnetic field from the previous time
step. The term $-\bm{\beta} \times (\nabla \times \bm{E})$ in
\eqref{eq:E-cd} is the electric convection, which is an analogue of
the $\bm{\beta} \cdot \nabla u$ in the scalar convection-diffusion
equation
\begin{equation} \label{eq:scalar-cd} 
-\nabla\cdot(\alpha \nabla u) + \bm{\beta}\cdot \nabla u + \gamma u =
f.
\end{equation} 
The theory and numerical analysis of such a scalar
convection-diffusion equation are well-studied in the literature.  It
is well known that, for small $\alpha$, boundary layers may appear in
the solution of \eqref{eq:scalar-cd} and standard finite element
methods may suffer from strong numerical oscillations and
instabilities if the mesh size is not small enough.

%% literature review on the stabilized method
Numerous studies on the stable discretization of scalar
convection-diffusion have been published. In the finite element
methods, various special strategies have been developed,
including the stabilized discontinuous Galerkin method
\cite{houston2002discontinuous, brezzi2004discontinuous}, SUPG method
\cite{brooks1982streamline, franca1992stabilized,
burman2010consistent}, bubble function stabilization
\cite{brezzi1994choosing, brezzi1998applications, brezzi1998further,
franca2002stability}, local projection stabilization
\cite{ganesan2010stabilization}, edge stabilization and continuous
interior penalty method \cite{burman2004edge,
burman2005unified, burman2007continuous}. 
%, discontinuous Galerkin methods \cite{brezzi2004discontinuous}.  
Other studies do not require the characteristics to be specified,
such as exponential fitting \cite{brezzi1989two, xu1999monotone,
dorfler1999uniform, lazarov2012exponential, bank2017arbitrary} and
Petrov-Galerkin method \cite{morton2009convection,
christiansen2014stability}.

%% literature review of the H(curl), H(div)
The $H(\curl)$ and $H(\div)$ convection-diffusion problems have
received more and more  attention from numerical computation. The
discretization of the general convection, known as extrusion, has been
discussed via Whitney forms in \cite{bossavit2003extrusion}.
For the pure advection problem, the stabilized Galerkin method has
been extended from $0$-form \cite{brezzi2004discontinuous} to $1$-form
\cite{heumann2013stabilized} and $k$-form \cite{heumann2011eulerian,
heumann2015stabilized}. These discretizations of the advection problem,
along with the proper discretization of the diffusion term, are feasible to
tackle the general convection-diffusion problems. Besides the Eulerian
method, the semi-Lagrangian method can be applied to the
time-dependent convection-diffusion problems for differential forms
\cite{heumann2011eulerian, heumann2012fully, heumann2013convergence}.

%% motivation 
More specifically, we are motivated by the Edge-Averaged Finite
Element (EAFE) method for scalar convection-diffusion problem proposed
by Xu and Zikatanov \cite{xu1999monotone}. There are two main
advantages to using EAFE: (1) The monotonicity of EAFE can be
established for a very general class of meshes; (2) The local
stiffness matrix of EAFE can easily be obtained by modifying that of
standard Poisson. A construction that ensures the general SPD
diffusion coefficient matrix was proposed in
\cite{lazarov2012exponential}. A high-order Scharfetter-Gummel scheme,
known as a high-order extension of EAFE, was given in
\cite{bank2017arbitrary}.  Similar to \eqref{eq:scalar-cd}, the
standard finite element methods also seriously suffer from numerical
instabilities for \eqref{eq:E-cd} for a large magnetic Reynolds number
$R_{m}$.  In this paper, we extend the EAFE scheme
\cite{xu1999monotone} to the $H(D)$ convection-diffusion equations so
that the resulting finite element discretizations work for a wide
range of diffusion coefficients $\alpha$.  

%% our contribution 
The proposed schemes for $H(D)$ convection-diffusion problems have
several intriguing features.  First, thanks to the special properties
of the $\mathcal{P}_1^-\Lambda^k$ discrete de Rham complex, the
schemes are the standard variational formulations modified by properly
averaging the PDE coefficients on the sub-simplexes, and are therefore
named {\it simplex-averaged finite element (SAFE)} schemes. Second,
their derivations stem from the graph Laplacian for $H(D)$ diffusion,
where only $D = \grad$ was given in the previous literature.  Third,
by introducing several special interpolations $\bar{\Pi}^k_T$, the
schemes can be recast into the equivalent ones that are suitable for
the analysis. Last, by means of the Bernoulli functions, the resulting
schemes are shown to converge to special upwind schemes as the
diffusion coefficient approaches to zero. The SAFE schemes also
provide a promising way to discrete the Lie convection with Hodge
Laplacian \cite{arnold2017finite}.

%% rest of the paper
The rest of the paper is organized as follows. In Section
\ref{sec:preliminaries} we introduce the general convection-diffusion
problems and briefly review the $\mathcal{P}_1^-\Lambda^k$
discrete de Rham complex. In Section \ref{sec:local-operator} we give
a crucial identity which holds for $\grad$, $\curl$, and $\div$ in a
unified fashion, then introduce local simplex-averaged operators. In
Section \ref{sec:SAFE} we derive the simplex-averaged finite element
schemes for $H(D)$ convection-diffusion problems. An important step
here is the derivation of $H(D)$ graph Laplacian.  In Section
\ref{sec:analysis} we prove the stability of SAFE for sufficiently
small mesh size and establish the error estimate under minimal
smoothness of the solution. Finally, in Section \ref{sec:numerical},
we show that the SAFE schemes are robust and effective for general
convection-diffusion problems through numerical tests. The detailed
implementation and limiting schemes are presented in Appendix
\ref{sec:append}.

%%%% Preliminaries %%%%
\section{Preliminaries} \label{sec:preliminaries}
In this section, we introduce some notation and briefly review some
basic properties of finite element triangulations and finite element
spaces.  In particular, we discuss some special properties of the
$\mathcal{P}_1^-\Lambda^k$ discrete de Rham complex which, as we shall
see later, will be the basis of devising the SAFE schemes for $H(D)$
convection-diffusion problems.

Let the domain $\Omega$ is a bounded polyhedron in $\mathbb{R}^\ell$
($\ell=2,3$).  Given $p\in [1,\infty]$ and an integer $m\geq 0$, we use
the usual notation $W^{m,p}(\Omega), \|\cdot\|_{m,p,\Omega},
|\cdot|_{m,p,\Omega}$ to denote the usual Sobolev space, norm and
semi-norm, respectively.  When $p = 2$, $H^m(\Omega):=
W^{m,p}(\Omega)$ with $|\cdot|_{m,\Omega}:= |\cdot|_{m,2,\Omega}$ and
$\|\cdot\|_{m,\Omega} = \|\cdot\|_{m,2,\Omega}$.  Let $\mathcal{T}_h$
be a conforming and shape-regular triangulations of $\Omega$. $h_T$ is
the diameter of $T$, and $h := \max_{T\in \mathcal{T}_h} h_T$.

Throughout this paper, we assume the dimension $\ell=3$, although all
the results extend without major modifications to the case in which
$\ell=2$.

%% Model problem %%
\subsection{Model problems}
\commentone{
Given a vector field $\beta(x)$, in this paper, we consider the
general convection-diffusion problem in the following three forms: 
\begin{subequations}
\begin{enumerate}
\item $H(\grad)$ convection-diffusion problem: 
\begin{equation} \label{equ:grad-conservative}
\left\{
\begin{array}{ll}
-\div(\alpha \nabla u + \beta u) + \gamma u = f &\quad \text{in }
\Omega,\\ 
u = 0 & \quad \text{on }\Gamma_0 \subset \partial \Omega, \\ 
(\alpha \nabla u + \beta u)\cdot n = g & \quad \text{on }
\Gamma_{N} = \partial \Omega \setminus \Gamma_0. \\ 
\end{array}
\right.
\end{equation} 
\item $H(\curl)$ convection-diffusion problem: 
\begin{equation} \label{equ:curl-conservative}
\left\{
\begin{array}{ll}
\nabla \times(\alpha \nabla \times u + \beta \times u) + \gamma u = f
&\quad \text{in }
\Omega,\\ 
n \times u = 0 & \quad \text{on }\Gamma_0 \subset \partial \Omega, \\ 
n \times (\alpha \nabla \times u + \beta \times u) = g & \quad \text{on }
\Gamma_{N} = \partial \Omega \setminus \Gamma_0. \\ 
\end{array}
\right.
\end{equation} 
\item $H(\div)$ convection-diffusion problem: 
\begin{equation} \label{equ:div-conservative}
\left\{
\begin{array}{ll}
-\nabla (\alpha \nabla \cdot u + \beta \cdot u) + \gamma u = f
&\quad \text{in }
\Omega,\\ 
u \cdot n = 0 & \quad \text{on }\Gamma_0 \subset \partial \Omega, \\ 
\alpha \nabla \cdot u + \beta \cdot u = g & \quad \text{on }
\Gamma_{N} = \partial \Omega \setminus \Gamma_0. \\ 
\end{array}
\right.
\end{equation} 
\end{enumerate}
\end{subequations}
Here, $n$ is the unit outer vector normal to $\partial \Omega$. 
To allow a parallel treatment of the above forms, we unify the
presentation of
\eqref{equ:grad-conservative}--\eqref{equ:div-conservative} as follows 
}
\begin{equation} \label{equ:conv-diff}
\left\{
\begin{array}{ll}
\mathcal{L}u := d^*(\alpha du + i^*_{\beta} u) + \gamma u = f
&\quad \text{in } \Omega,\\ 
\tr(u) = 0 & \quad \text{on }\Gamma_0 \subset \partial \Omega, \\ 
\tr[\star (\alpha du + i_{\beta}^*u)] = g & \quad \text{on }
\Gamma_{N} = \partial \Omega \setminus \Gamma_0. \\ 
\end{array}
\right.
\end{equation}
\commentone{  
Here, the unknown $u$ is a vector proxy of differential $k$-form in
3D. In terms of vector proxy in 3D, $d = {\rm grad}$ (or $\nabla$)
when $k=0$, $d = {\rm curl}$ (or $\nabla \times$) when $k=1$, and $d =
{\rm div}$ (or $\nabla \cdot$) when $k=2$.  $d^*$, $i_\beta$,
$i_\beta^*$, $\star$ and $\tr$ denote the vector proxy of
coderivative, contraction, dual of contraction (or the limiting of
extrusion \cite{bossavit2003extrusion, heumann2011eulerian}),
Hodge star, and trace operator in 3D, respectively (cf.
\cite{arnold2006finite}). The correspondences between the exterior
calculus notations and the expressions for vector proxies
can be easily obtained by comparing
\eqref{equ:grad-conservative}-\eqref{equ:div-conservative} and 
\eqref{equ:conv-diff}, and are summarized in Table \ref{tab:proxy-3D}.
}  

\begin{table}[!htbp]
\centering 
\commentone{
\begin{tabular}{m{0.6cm}m{3.0cm}m{3.0cm}m{1.2cm}m{1.2cm}m{1.2cm}}
\hline  
$k$ & $du$ & $d^*u$ & $i_{\beta}u$ & $i_{\beta}^* u$ & $\mathrm{tr}$\\
\hline\hline
$0$ & $\grad u$ (or $\nabla u$) & $-\div u$ (or $-\nabla \cdot u$) & & $\beta u$ & $u$ \\ 
$1$ & $\curl u$ (or $\nabla \times u$) & $\curl u$ (or $\nabla \times
    u$) & $\beta\cdot u$ &
$\beta\times u$ & $n\times u$ \\ 
$2$ & $\div u$ (or $\nabla \cdot u$) & $-\grad u$ (or $-\nabla u$) & $-\beta\times u$ &
$\beta\cdot u$ & $u\cdot n$ \\ 
$3$ & & & $\beta u$ &
&  \\ \hline 
\end{tabular}
}
\caption{\commentone{Translation table for unifying notational
  framework.}}
\label{tab:proxy-3D}
\end{table}

We also consider the following boundary value problems that are
associated with the dual of the operator $\mathcal{L}$ in
\eqref{equ:conv-diff}:
\begin{equation} \label{equ:conv-diff-dual}
\left\{
\begin{array}{ll}
\mathcal{L}^*u := d^*(\alpha du) + i_{\beta} du + \gamma u = f
&\quad \text{in } \Omega,\\ 
\tr(u) = 0 & \quad \text{on }\Gamma_0 \subset \partial \Omega, \\ 
\tr[\star (\alpha du)] = g & \quad \text{on }
\Gamma_{N} = \partial \Omega \setminus \Gamma_0. \\ 
\end{array}
\right.
\end{equation}
\commentone{Note that the model problem \eqref{equ:conv-diff-dual} in
3D corresponds to \eqref{eq:scalar-cd} and \eqref{eq:E-cd} when $k=0$
and $k=1$, respectively.}

For both of the above model problems, we assume that $\Gamma_0$ has
positive surface measure. The coefficients are assumed to satisfy
$\alpha(x)\in W^{1,\infty}(\Omega;\mathbb{R})$,
$\beta(x)=(\beta_i(x))\in W^{1,\infty}(\Omega; \mathbb R^n)$ and
$\gamma(x) \in L^\infty(\Omega; \mathbb{R})$. We further assume that
$\alpha(x)$ and $\gamma(x)$ are uniformly positive, i.e.,
$$ 
0 < \alpha_0 \leq \alpha(x) \leq \alpha_1\quad \text{and} \quad
0 < \gamma_0 \leq \gamma(x).
$$ 

Define the space
$$ 
V := \{w \in H\Lambda^k(\Omega): ~ \tr(w)= 0~\text{on } \Gamma_0\},
$$ 
equipped with the norm $\|w\|_{H\Lambda,\Omega}^2 :=
\|w\|_{0,\Omega}^2 + \|dw\|_{0,\Omega}^2$.  Then, the variational
formulation for \eqref{equ:conv-diff} is: Find $u \in V$ such that 
\begin{equation} \label{equ:variational} 
a(u, v) = F(v) \qquad \forall v\in V,
\end{equation} 
where 
$$ 
a(u, v) := (\alpha du + i_{\beta}^* u, dv) + (\gamma u, v), \quad
F(v) := (f, v) + \langle g, \tr(v) \rangle_{\Gamma_N}.
$$ 
And the variational formulation for \eqref{equ:conv-diff-dual} is:
Find $u \in V$ such that 
\begin{equation} \label{equ:variational-dual}
a^*(u, v) = F(v) \qquad \forall v\in V,
\end{equation}
where $a^*(u,v) := a(v, u) = (du, \alpha dv + i_{\beta}^*v) + (\gamma
u, v)$.  We make the following assumptions for the well-posedness of 
convection-diffusion problems \eqref{equ:conv-diff} and
\eqref{equ:conv-diff-dual}.
\begin{assumption}[Well-posedness]\label{asp} The operators 
$$
\mathcal{L}, \mathcal{L}^*: V \mapsto V' 
$$ 
are isomorphisms. Namely both \eqref{equ:conv-diff} and
\eqref{equ:conv-diff-dual} are uniquely solvable. Furthermore, there
exists a constant $c_0>0$ (which may depend on $\alpha$, $\beta$,
$\gamma$) such that 
\begin{equation} \label{equ:inf-sup}
\inf_{u \in V}\sup_{v\in V} 
\frac{a(u,v)}{\|u\|_{H\Lambda,\Omega}\|v\|_{H\Lambda,\Omega}}
= \inf_{u\in V}\sup_{v\in V}\frac{a^*(u,v)}{\|u\|_{H\Lambda,\Omega}
\|v\|_{H\Lambda,\Omega}} = c_0>0.
\end{equation}
\end{assumption}

\begin{remark} The above assumption holds for $H^1$
convection-diffusion problem by using the weak maximum principle (cf.
\cite[Section 8.1]{gilbarg2015elliptic}) and Fredholm alternative
theory (cf. \cite[Theorem 4, pp. 303]{evans2010partial}). A sufficient
condition for the above assumption is that $4\alpha(x)\gamma(x) \geq
|\beta(x)|_{l^2}^2$ for all $x \in \Omega$, \commentone{by using the
Cauchy-Schwarz inequality}.
\end{remark}

\begin{remark}
In \cite{heumann2013stabilized, heumann2015stabilized}, the
\commentone{proxy of} Lie convection problems considered as the model
problems. Thanks to the theory of Friedrichs' symmetric operators
\cite{friedrichs1958symmetric}, a sufficient condition that depends
only on $\beta(x)$ and $\gamma(x)$ can be given for the purpose of
coercivity. In this paper, we only consider the model problems
\eqref{equ:conv-diff} and \eqref{equ:conv-diff-dual}, which are the
simplest ones to present the features of SAFE schemes. The SAFE
schemes for the \commentone{proxy of Lie convection and their
applications} will be reported in the subsequent work.
\end{remark}

%% de Rham %%
\subsection{$\mathcal{P}_1^-\Lambda^k$ discrete de Rham complex}
In this paper, we confine to the $\mathcal{P}_1^-\Lambda^k$ discrete
de Rham complex,
i.e.
\begin{equation} \label{equ:discrete-complex} 
\mathcal{P}_1^-\Lambda^0 \xrightarrow{\grad} \mathcal{P}_1^- \Lambda^1
\xrightarrow{\curl} \mathcal{P}_1^-\Lambda^2 \xrightarrow{\div}
\mathcal{P}_1^-\Lambda^3.
\end{equation} 
The local basis functions of $\mathcal{P}_1^-\Lambda^k(T)$, which are
associated with the sub-simplexes of $T$, are denoted by
$\varphi_{a}$, $\varphi_E$, $\varphi_F$ and $\varphi_T$, respectively.
The local degrees of freedom satisfy (\commentone{see Fig.
\ref{fig:dofs}}) 
$$ 
\begin{aligned}
l_{a}^0(\varphi_{a'}) = \varphi_{a'}(a) = \delta_{aa'}, &
\qquad l_{E}^1(\varphi_{E'}) = \int_{E} \varphi_{E'} \cdot \tau_{E} =
\delta_{EE'}, \\ 
l_{F}^2(\varphi_{F'}) = \int_{F} \varphi_{F'} \cdot n_{F} =
\delta_{FF'}, & 
\qquad l_{T}^3(\varphi_{T'}) = \int_{T} \varphi_{T'} = \delta_{TT'}. 
\end{aligned}
$$ 
\vspace{-4mm}
\begin{figure}[!htbp]
  \centering
  \includegraphics[width=0.20\textwidth]{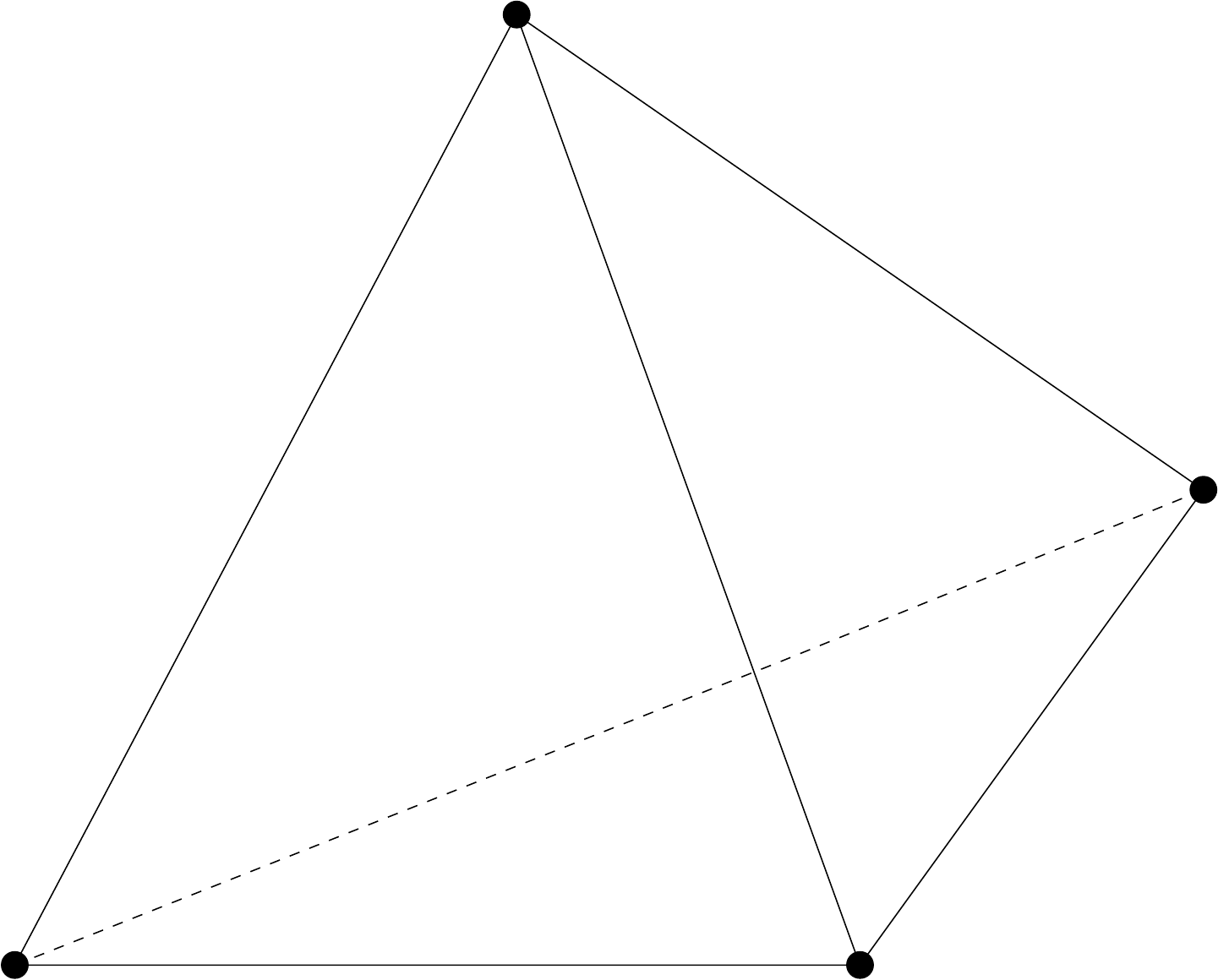}
  \includegraphics[width=0.20\textwidth]{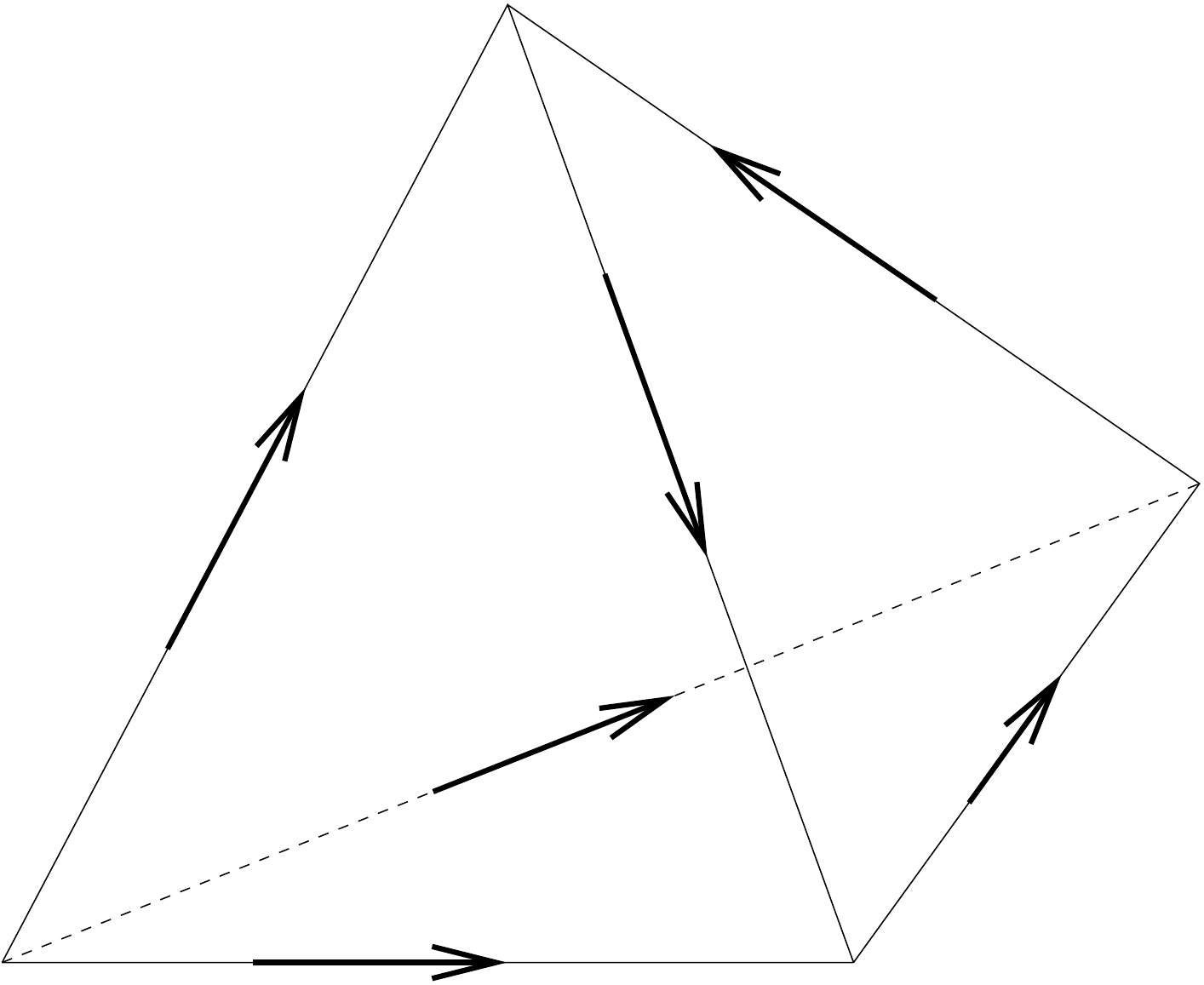}
  \includegraphics[width=0.20\textwidth]{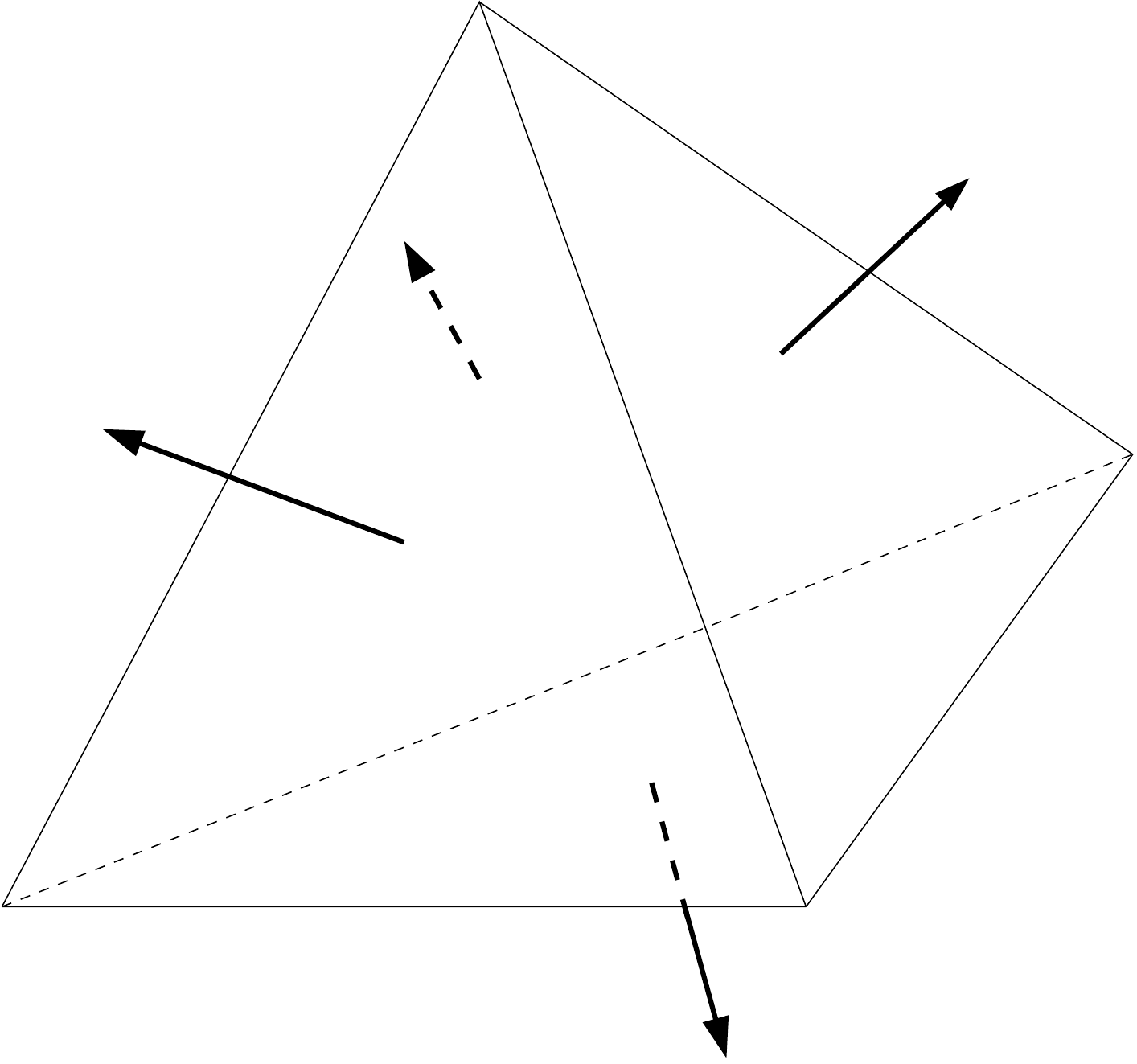}
  \includegraphics[width=0.20\textwidth]{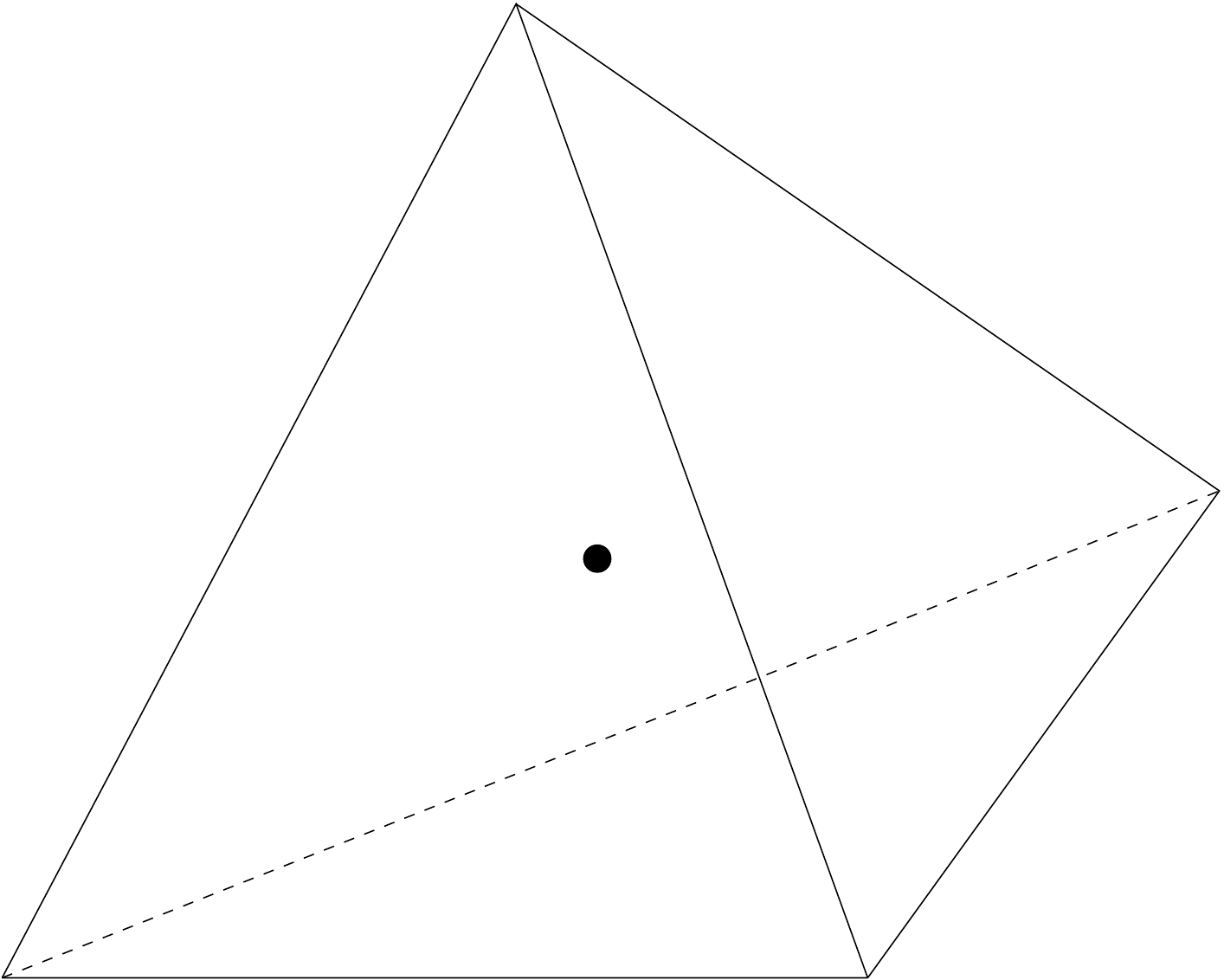}
  \caption{\commentone{Symbolic notation for local degrees of freedom for
$\mathcal{P}_1^-\Lambda^0$, $\mathcal{P}_1^-\Lambda^1$,
  $\mathcal{P}_1^-\Lambda^2$, and $\mathcal{P}_1^{-}\Lambda^3$ (left
to right)}}
  \label{fig:dofs}
\end{figure}
\vspace{-4mm}

Denote by $\mathcal{S}_T^k$ the set of sub-simplexes of dimension $k$.
Thus, the set local degrees of freedom of
$\mathcal{P}_1^-\Lambda^k(T)$ can be written as $\{l_S^k(\cdot)~|~
S\in \mathcal{S}_T^k \}$.  Then, the local canonical interpolation
operator can be written as 
\begin{equation} \label{equ:interpolate}
\Pi_T^k v:= \sum_{S\in \mathcal{S}_T^k} l_S^k(v) \varphi_S. 
\end{equation}
We also define $\delta_S(v) = l_S^{k+1}(d v)$ for any $v\in
H\Lambda^k(T)$ and $S\in \mathcal{S}_T^{k+1}$.

%%%% Local operator %%%% 
\section{\commentone{Local Discretization of Convection-diffusion
Operators}} \label{sec:local-operator}
In this section, we explain the idea of exponential fitting and
construct the local simplex-averaged operators. 

\subsection{A crucial identity}
Let $\theta = \beta/\alpha$.  We first consider the case in which
$\theta$ is a constant. Let $J_{\theta}^ku = d^k u + i_{\theta}^*u$.
In \cite{lazarov2012exponential}, it is shown that, when $k=0$,
$$ 
J_{\theta}^{0}u = \nabla u + \theta u = \exp(-\theta \cdot x)
\nabla \left[ \exp(\theta \cdot x)u \right],
$$ 
which motivates the following lemma serving as the starting point of
this paper.
\begin{lemma} \label{lem:exp-fitting-k} 
Assume that $\theta$ is a constant vector. It holds that 
\begin{equation} \label{equ:exponential-fitting-k}
J_{\theta}^{k}u =  d^k u + i_{\theta}^*u =
\exp(-\theta \cdot x) d^k \left[
\exp(\theta \cdot x) u \right].
\end{equation}
\end{lemma}
\begin{proof}
We prove \eqref{equ:exponential-fitting-k} case by case:
\begin{enumerate}
\item $k=0$ and $d^0 = \nabla$. It is straightforward to show that 
$$ 
\begin{aligned}
\text{R.H.S of \eqref{equ:exponential-fitting-k}} &= \exp(-\theta \cdot x) 
\begin{pmatrix}
\exp(\theta\cdot x) \partial_{x_1}u + \theta_1 \exp(\theta\cdot x)
  u \\ 
\exp(\theta\cdot x) \partial_{x_2}u + \theta_2 \exp(\theta\cdot x)
  u \\ 
\exp(\theta\cdot x) \partial_{x_3}u + \theta_3 \exp(\theta\cdot x)
  u 
\end{pmatrix} \\
&= \nabla u + \theta u.
\end{aligned}
$$ 
\item $k=1$ and $d^1 = \nabla \times$. Then, a direct calculation
shows that 
$$ 
\begin{aligned}
\text{R.H.S of \eqref{equ:exponential-fitting-k}} &= \exp(-\theta\cdot x) 
\begin{pmatrix} 
\partial_{x_2}[\exp(\theta\cdot x)u_3] -
\partial_{x_3}[\exp(\theta\cdot x)u_2] \\
\partial_{x_3}[\exp(\theta\cdot x)u_1] -
\partial_{x_1}[\exp(\theta\cdot x)u_3] \\
\partial_{x_1}[\exp(\theta\cdot x)u_2] -
\partial_{x_2}[\exp(\theta\cdot x)u_1] 
\end{pmatrix} \\ 
&= 
\begin{pmatrix}
\partial_{x_2}u_{3} - \partial_{x_3}u_{2} \\
\partial_{x_3}u_{1} - \partial_{x_1}u_{3} \\
\partial_{x_1}u_{2} - \partial_{x_2}u_{1} 
\end{pmatrix} +
\begin{pmatrix}
\theta_2u_3 - \theta_3 u_2 \\ 
\theta_3u_1 - \theta_1 u_3 \\
\theta_1u_2 - \theta_2 u_1 
\end{pmatrix} \\ 
&= \nabla \times u + \theta \times u. 
\end{aligned}
$$ 
\item $k=2$ and $d^2 = \nabla \cdot$. Clearly, 
$$ 
\text{R.H.S of \eqref{equ:exponential-fitting-k}} = \exp(-\theta \cdot
x)\sum_{i} \partial_{x_i}[\exp(\theta \cdot x)u_i] = \nabla
\cdot u + \theta\cdot u.
$$ 
\end{enumerate}
This completes the proof.
\end{proof}
\begin{remark}
The above lemma is a special case of the gauge theory in differential
geometry, see \cite{eguchi1980gravitation, christiansen2013upwinding}.
\end{remark}

Define the operator $E_{\theta}$ by $E_{\theta}u = \exp(\theta\cdot
x)u$.  Thanks to Lemma \ref{lem:exp-fitting-k}, we have the following
commutative diagram when $\theta$ is constant: 
\begin{equation} \label{equ:diagram1}
\begin{tikzcd}
C^\infty(\Omega) \arrow[r, "\grad"] 
\arrow[d, shift right, swap, "E_{-\theta}"] 
& C^\infty(\Omega;\mathbb{R}^3) \arrow[r, "\curl"] 
\arrow[d, shift right, swap, "E_{-\theta}"] 
& C^\infty(\Omega;\mathbb{R}^3) \arrow[r,"\div"] 
\arrow[d, shift right, swap, "E_{-\theta}"] 
& C^\infty(\Omega) 
\arrow[d, shift right, swap, "E_{-\theta}"] \\ 
C^\infty(\Omega) \arrow[r, "J_{\theta}^0"] 
\arrow[u, shift right, swap, "E_{\theta}"] 
& C^\infty(\Omega;\mathbb{R}^3) \arrow[r, "J_{\theta}^1"] 
\arrow[u, shift right, swap, "E_{\theta}"] 
& C^\infty(\Omega;\mathbb{R}^3) \arrow[r, "J_{\theta}^2"] 
\arrow[u, shift right, swap, "E_{\theta}"] 
& C^\infty(\Omega) 
\arrow[u, shift right, swap, "E_{\theta}"] 
\end{tikzcd}
\end{equation}
We note that a useful feature of the above commutative diagram is the
invariance against spatial translation. Namely, \eqref{equ:diagram1}
also holds when defining $E_{\theta}u = \exp(\theta\cdot (x-x_0))$ for
any $x_0 \in \mathbb{R}^n$. 

\subsection{Local simplex-averaged operators}
\commentone{In the spirit of the exponentially fitting scheme, Lemma
\ref{lem:exp-fitting-k} builds a foundation in designing a robust
scheme with the convection-dominated region.} To this end, first we
explain the simplex-averaged operators on an element $T$. Thanks to
the commutativity property that $d^k \Pi_T^k = \Pi_T^{k+1}d^k$ and
Lemma \ref{lem:exp-fitting-k}, we formally obtain 
$$
\Pi_{T}^{k+1}[\exp(\theta \cdot x) J_{\theta}^{k}u] 
= \Pi_{T}^{k+1} d^k [ \exp(\theta\cdot x) u] 
= d^k \Pi_{T}^{k} [ \exp(\theta\cdot x) u]. 
$$
Therefore, we define the operator $J_{\theta,T}^{k}$ that mimics the
above equality at discrete level.  

\begin{definition} \label{def:J-T}
The local operator $J_{\theta,T}^{k}: \mathcal{P}_1^-\Lambda^k(T)
\mapsto \mathcal{P}_1^-\Lambda^{k+1}(T)$ is defined by  
\begin{equation} \label{equ:exponential-flux}
\Pi_{T}^{k+1}[\exp(\theta \cdot x) J_{\theta,T}^{k}w_h] := d^k 
\Pi_{T}^k[\exp(\theta \cdot x) w_h] \qquad \forall w_h
\in \mathcal{P}_1^-\Lambda^{k}(T).
\end{equation} 
\end{definition}
In order to show the well-posedness of Definition \ref{def:J-T},
we first show the well-posedness of the {\it simplex-averaged
operator} given below. 
\begin{definition}
The simplex-averaged operator $H_{\theta,T}^{k}:
\mathcal{P}_1^-\Lambda^k(T) \mapsto \mathcal{P}_1^-\Lambda^k(T)$ is
defined by 
\begin{equation} \label{equ:s-avg}
H_{\theta,T}^k w_h = \sum_{S\in \mathcal{S}_T^k} \left( 
\dashint_S \exp(\theta \cdot x)
\right)^{-1} l_S^k(w_h) \varphi_S \qquad \forall w_h \in
\mathcal{P}_1^-\Lambda^k(T),
\end{equation} 
where $\dashint_S$ is the average integral on $S \in \mathcal{S}_T^k$.
\end{definition}

\begin{lemma} \label{lem:s-avg}
It holds that $H_{\theta,T}^k =
\left(\Pi_T^{k}E_{\theta}\right)^{-1}$ on
$\mathcal{P}_1^-\Lambda^k(T)$. 
\end{lemma}
\begin{proof}
Note that the basis functions of $\mathcal{P}_1^-\Lambda^k(T)$ satisfy 
$$ 
\varphi_{a}(a') = \delta_{aa'}, \qquad \varphi_{E}\cdot
\tau_{E'} = \frac{\delta_{EE'}}{|E'|}, \qquad \varphi_F\cdot n_{F'} =
\frac{\delta_{FF'}}{|F'|}, \qquad \varphi_T = \frac{1}{|T|}. 
$$ 
Therefore,  for any $w_h \in \mathcal{P}_1^-\Lambda^k(T)$,
$$ 
\begin{aligned}
\Pi_T^k(E_{\theta} w_h) &=
\sum_{S' \in \mathcal{S}_T^k}  l_{S'}^k \left( \exp(\theta\cdot x)
\sum_{S\in \mathcal{S}_T^k} l_S^k(w_h) \varphi_S \right) \varphi_{S'}
\\
&= \sum_{S' \in \mathcal{S}_T^k} \left( \dashint_{S'}
\exp(\theta\cdot x) \right) l_{S'}^k(w_h) \varphi_{S'},
\end{aligned}
$$ 
which implies that $H_{\theta,T}^k\Pi_T^k E_{\theta}w_h = w_h$.
\end{proof}
In light of Lemma \ref{lem:s-avg}, $J_{\theta,T}^k$ in
\eqref{equ:exponential-flux} can be written explicitly in terms of
the simplex-averaged operator 
\begin{equation} \label{equ:exp-flux2} 
J_{\theta,T}^k = \left(\Pi_T^{k+1}E_{\theta}\right)^{-1} d^k
\Pi_T^k E_{\theta} = H_{\theta,T}^{k+1} d^k \Pi_T^k E_{\theta}.
\end{equation}
Further, we can define the interpolations
$\tilde{\Pi}_{\theta,T}^k: \Lambda^k(T) \mapsto
\mathcal{P}_1^-\Lambda^k(T)$ by 
\begin{equation} \label{equ:tilde-proj}
\tilde{\Pi}_{\theta,T}^k v := H_{\theta,T}^k \Pi_T^k E_{\theta}v 
= \sum_{S \in \mathcal{S}_T^k} \frac{l_S^k(\exp(\theta\cdot
x)v)}{\dashint_S \exp(\theta\cdot x)} \varphi_S.
\end{equation}
In summary, we depict the 3D-commutative diagram in Fig.
\ref{fig:3D-commutative}. The exactness of discrete de Rham complex 
and Lemma \ref{lem:s-avg} lead to the following corollary.  

\begin{figure}
\begin{center}
\begin{tikzpicture}[node distance=11.66em,
  back line/.style={densely dotted},
  cross line/.style={preaction={draw=white, -,line width=6pt}}]
  \node (Jk) {${\Lambda}^k(T)$};
  \node [right of=Jk] (Jk1) {${\Lambda}^{k+1}(T)$};
  \node [below of=Jk] (DJk) {$\mathcal{P}_1^-\Lambda^k(T)$};
  \node [right of=DJk] (DJk1) {$\mathcal{P}_1^-\Lambda^{k+1}(T)$};
  
  \draw[->] (Jk) to node [above] {\scriptsize $J_{\theta}^k$} (Jk1); 
  \draw[->] (DJk) to node [above] {\scriptsize $J_{\theta,T}^k$} (DJk1); 
  \draw[->] (Jk) to node [left] {$\scriptsize \tilde{\Pi}_{\theta,T}^k$} (DJk);
  \draw[->] (Jk1) to node [right] {$\scriptsize \tilde{\Pi}_{\theta,T}^{k+1}$} (DJk1);
 
  \node (Hk) [right of=Jk, above of=Jk, node distance=5em]
{$\Lambda^k(T)$};
  \node [right of=Hk] (Hk1) {$\Lambda^{k+1}(T)$};
  \node [below of=Hk] (DHk) {$\mathcal{P}_1^-\Lambda^k(T)$};
  \node [right of=DHk] (DHk1) {$\mathcal{P}_1^-\Lambda^{k+1}(T)$};

  \draw[dashed,->] (DHk) to node [above] {\scriptsize $d^k$} (DHk1); 
  \draw[dashed,->] (Hk) to node [above] {\scriptsize $d^k$} (Hk1); 
  \draw[dashed,->] (Hk) to node [left] {\scriptsize $\Pi_T^k$} (DHk);
  \draw[dashed,->] (Hk1) to node [right] {\scriptsize $\Pi_T^{k+1}$} (DHk1);

  \draw[->] (Hk.220) to node [left] {\scriptsize $E_{-\theta}$} (Jk.50);
  \draw[->] (Jk.33) to node [right] {\scriptsize $E_{\theta}$} (Hk.240);
  \draw[->] (Hk1.220) to node [left] {\scriptsize $E_{-\theta}$} (Jk1.50);
  \draw[->] (Jk1.33) to node [right] {\scriptsize $E_{\theta}$} (Hk1.240);

  \draw[->] (DHk.220) to node [left] {\scriptsize $H_{\theta,T}^k$} (DJk.50);
  \draw[->] (DJk.33) to node [right] {\scriptsize $\Pi_T^{k}E_{\theta}$} (DHk.240);
  \draw[->] (DHk1.220) to node [left] {\scriptsize $H_{\theta,T}^{k+1}$} (DJk1.50);
  \draw[->] (DJk1.33) to node [right] {\scriptsize $\Pi_T^{k+1}E_{\theta}$} (DHk1.240);
\end{tikzpicture}
\end{center}
\caption{3D-commutative diagram, the front and above diagrams require
$\theta$ to be constant.}
\label{fig:3D-commutative}
\end{figure}
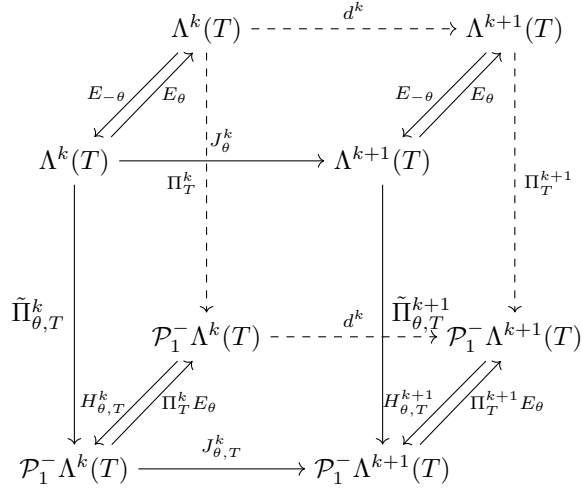

\begin{corollary}
It holds that $J_{\theta,T}^{k+1}J_{\theta,T}^k = 0$.
\end{corollary}

%%%% SAFE %%%%
\section{Simplex-averaged Finite Element Methods} \label{sec:SAFE}
In this section, we present a family of finite element approximations
for \eqref{equ:conv-diff} and \eqref{equ:conv-diff-dual}.

Thanks to the $J_{\theta,T}$ given in Definition \ref{def:J-T}, first
we introduce the following local bilinear form on a fixed element
$T\subset \mathcal{T}_h$ 
\begin{equation} \label{equ:bilinear-bT}
\tilde{a}_{T}(w_h, v_h) := (\alpha J_{\theta,T}^k w_h, d^k v_h)_T 
\quad \forall w_h, v_h \in \mathcal{P}_1^-\Lambda^k(T).
\end{equation}
We give the explicit form of \eqref{equ:bilinear-bT} in the following
theorem. 
\begin{theorem} \label{thm:bT2}
It holds that, for any $w_h, v_h \in \mathcal{P}_1^-\Lambda^k(T)$, 
\begin{equation} \label{equ:bT} 
\tilde{a}_{T}(w_h, v_h) = \sum_{S \in \mathcal{S}_T^{k+1}}
\left( \dashint_S \exp(\theta\cdot x) \right)^{-1}
\delta_S\left( 
\exp(\theta\cdot x) w_h 
\right) \commentone{(\alpha \varphi_S, d^k v_h)_T}. 
%\int_T \alpha \varphi_S \land \star d^kv_h.
\end{equation} 
\end{theorem}
\begin{proof}
In light of \eqref{equ:exp-flux2} and commutativity property, we have 
$$ 
\begin{aligned}
J^k_{\theta,T} w_h &= H_{\theta,T}^{k+1} d^k \Pi_T^k
E_{\theta}w_h = H_{\theta,T}^{k+1} \Pi_T^{k+1} d^k E_{\theta}w_h
\\
&= H_{\theta,T}^{k+1} \sum_{S\in \mathcal{S}_T^{k+1}}
l_S(d^kE_{\theta}w_h) \varphi_S \\
&= \sum_{S\in \mathcal{S}_T^{k+1}} \left( \dashint_S
\exp(\theta \cdot x) \right)^{-1}
\delta_S(E_{\theta}w_h)\varphi_S.
\end{aligned}
$$ 
Then, \eqref{equ:bT} follows from the definition of $\tilde{a}_T$ in
\eqref{equ:bilinear-bT}.
\end{proof}

Note that, \commentone{due to the term $(\alpha \varphi_S, d^kv_h)$
in \eqref{equ:bilinear-bT}}, the local bilinear form \eqref{equ:bT}
requires the local mass matrix of $\mathcal{P}_1^-\Lambda^{k+1}$.  In what follows,
we introduce the local bilinear form of SAFE which is more friendly to
the implementation. \commentone{More precisely, the local bilinear
form of SAFE only requires a modification of the local stiffness
matrix of $\mathcal{P}_1^-\Lambda^k$, see Appendix
\ref{subsec:Bernoulli} for the implementation issues}. The primary
step is to construct a local constant interpolation so that the
resulting bilinear form mimics the graph Laplacian. 

Let $\bar{\alpha}$ be the local $L^2$ projection of $\alpha$ on the
piecewise constant space.  Let $\bar{\theta}$ be a piecewise constant
approximation of $\theta$ such that  
\begin{equation} \label{equ:bar-theta}
\|\theta - \bar{\theta}\|_{0,\infty,T} \lesssim
h_T|\theta|_{1,\infty,T}.
\end{equation}
We also define the {\it harmonic average} on a sub-simplex $S \subset
\bar{T}$ as  
\begin{equation} \label{equ:harmonic-avg}
\mathcal{H}_S(\bar{\alpha}, \bar{\theta}) = \left( \dashint_S
\exp(\bar{\theta}\cdot x){\bar \alpha}^{-1} \right)^{-1}. 
\end{equation} 

%% H(grad) %%
\subsection{Local bilinear form of SAFE on
$\mathcal{P}_1^-\Lambda^0(T)$}
To make our point, we start from the well-known $H^1$ graph Laplacian 
\begin{equation} \label{equ:H1-graph}
(\grad w_h, \grad v_h)_T = \sum_{E\in \mathcal{S}_T^1} \omega_{E}^T
\delta_E(w_h)\delta_E(v_h) \qquad \forall w_h, v_h \in
\mathcal{P}_1^-\Lambda^0(T), 
\end{equation}
where
$\omega_{E}^T = -(\grad \varphi_{a_i}, \grad \varphi_{a_j})_T, 
E = \overrightarrow{a_i a_j}$, and $\tau_E = \frac{ \overrightarrow{a_i a_j}}{ |\overrightarrow{a_i
  a_j}| }$.
We have the following lemma. 
\begin{lemma} \label{lem:H1-identity}
The following identity holds
\begin{equation} \label{equ:H1-identity}
I = \sum_{E\in \mathcal{S}_T^1}
\omega_{E}^T \frac{|E|^2}{|T|} \tau_E \tau_E^T.
\end{equation}
\end{lemma}
\begin{proof}
\commentone{The proof follows by taking $w_h = \xi \cdot x$
and $v_h = \eta \cdot x$ in \eqref{equ:H1-graph} for arbitrary
$\xi,\eta \in \mathbb{R}^3$.}
%%For any $\xi, \eta \in \mathbb{R}^3$, let $u_\xi = \xi \cdot x$
%%and $u_\eta = \eta \cdot x$. Then, from \eqref{equ:H1-graph}, we
%%have  
%%$$ 
%%\begin{aligned}
%%|T| \xi^T \eta = \sum_{E\in \mathcal{S}_T^1} \omega_{E}^T
%%\delta_E(u_\xi)\delta_E(u_\eta) &= \sum_{E\in \mathcal{S}_T^1}
%%\omega_{E}^T |E|^2(\xi \cdot \tau_E) (\eta\cdot \tau_E) \\
%%&= \xi^T \left( \sum_{E\in \mathcal{S}_T^1}
%%\omega_{E}^T |E|^2\tau_E\tau_E^T \right) \eta,
%%\end{aligned}
%%$$
%%which leads to \eqref{equ:H1-identity}.
\end{proof}
\begin{definition}
$\bar{\Pi}_T^1: \mathcal{P}_1^-\Lambda^1(T) \mapsto \mathbb{R}^3$ is
defined by 
\begin{equation} \label{equ:barPi-k1}
\bar{\Pi}_T^1 w_h := \sum_{E\in \mathcal{S}_T^1} \omega_E^T
\frac{|E|}{|T|} l_E(w_h) \tau_E \quad \forall w_h\in \mathcal{P}_1^-\Lambda^1(T).
\end{equation}
\end{definition}

\begin{lemma}
If $w_h$ is a constant vector on $T$, then $\bar{\Pi}_T^1 w_h = w_h$. 
\end{lemma}
\begin{proof}
If $w_h$ is constant, then $l_E(w_h) = |E|w_h \cdot \tau_E$. Thus, 
$$ 
\bar{\Pi}_T^1 w_h = \sum_{E\in \mathcal{S}_T^1} \omega_E^T
\frac{|E|}{|T|}|E|(w_h \cdot \tau_E) \tau_E = \left( 
\sum_{E\in \mathcal{S}_T^1} \omega_E^T \frac{|E|^2}{|T|}\tau_E\tau_E^T
\right) w_h = w_h.
$$ 
This completes the proof.
\end{proof}

We are now in the position to present the local bilinear form
for the $H^1$ convection-diffusion as 
\begin{equation} \label{equ:aT-k0}
a_T(w_h, v_h) := (\alpha \bar{\Pi}_T^1
J_{\bar{\theta},T}^0w_h, \grad v_h)_T \quad \forall w_h, v_h \in \mathcal{P}_1^-\Lambda^0(T).
\end{equation}

\begin{theorem}
It holds that, for any $w_h, v_h \in \mathcal{P}_1^-\Lambda^0(T)$, 
\begin{equation} \label{equ:aT2-k0}
a_T(w_h, v_h) = \sum_{E\in \mathcal{S}_T^1}
\omega_E^T \mathcal{H}_E(\bar{\alpha}, \bar{\theta})
\delta_E(\exp(\bar{\theta} \cdot x)w_h)\delta_E(v_h).
\end{equation}
\end{theorem}
\begin{proof}
From \eqref{equ:barPi-k1} and Theorem \ref{thm:bT2}, we have  
$$
\begin{aligned}
\bar{\Pi}_T^1 J_{\bar{\theta},T}^0w_h &=
\bar{\Pi}_T^1 \sum_{E\in \mathcal{S}_T^{1}} \left( \dashint_E
\exp(\bar{\theta} \cdot x) \right)^{-1}
\delta_E(E_{\bar{\theta}}w_h)\varphi_E \\
&= \sum_{E\in \mathcal{S}_T^{1}} \left( \dashint_E \exp(\bar{\theta}
\cdot x) \right)^{-1} \delta_E(E_{\bar{\theta}}w_h) \omega_E^T
\frac{|E|}{|T|}\tau_E.
\end{aligned}
$$ 
Therefore, 
$$ 
\begin{aligned}
a_T(w_h, v_h) &= \sum_{E\in \mathcal{S}_T^{1}} \left( \dashint_E
\exp(\bar{\theta} \cdot x) \right)^{-1}
\delta_E(E_{\bar{\theta}}w_h) \omega_E^T
(\alpha \frac{|E|}{|T|}\tau_E, \grad v_h)_T \\
&= \sum_{E\in \mathcal{S}_T^1} \omega_E^T
\mathcal{H}_E(\bar{\alpha}, \bar{\theta})
\delta_E(\exp(\bar{\theta} \cdot x)w_h)\delta_E(v_h).
\end{aligned}
$$ 
This completes the proof.
\end{proof}
\commentone{
We note that when $\theta$ is a local constant, the local bilinear form \eqref{equ:aT2-k0} for the $H(\grad)$ convection-diffusion problem 
coincides with the EAFE scheme (cf.\cite[Equ. (3.12)]{xu1999monotone}). The SAFE scheme for $H(\curl)$ and $H(\rm div)$ convection-diffusion problems below can be viewed as an extension of the EAFE scheme. 
}

%% H(curl) %%
\subsection{Local bilinear form of SAFE on
$\mathcal{P}_1^-\Lambda^1(T)$} By analogy an $H(\rm curl)$ graph
Laplacian is needed to construct the local constant projection
on $\mathcal{P}_1^-\Lambda^1(T)$. 

\begin{lemma} \label{lem:curl-graph}
For any $w_h, v_h \in \mathcal{P}_1^-\Lambda^{1}(T)$, it holds that 
\begin{equation} \label{equ:curl-graph}
(\curl w_h, \curl v_h)_T = \sum_{F,F'\in \mathcal{S}_T^2, F\neq F'}
\omega_{FF'}^T \delta_F(w_h)\delta_{F'}(v_h),
\end{equation}
where
$\omega_{FF'}^T = \omega_{F'F}^T = -\frac{1}{2} \|\curl \varphi_{\bar{F}\cap
\bar{F'}}\|_T^2$.
\end{lemma}
\begin{proof}
It suffices to show \eqref{equ:curl-graph} on the N\'ed\'elec basis
functions $\varphi_E = \varphi_{ij} := \lambda_i\nabla \lambda_j -
\lambda_j \nabla\lambda_i$ where $E = \overrightarrow{a_ia_j}$. That
is, $w_h = \varphi_E, v_h = \varphi_{E'}$. We consider the following
three cases (see Figure \ref{fig:3D-simplex}): 
\begin{figure}[!htbp]
  \centering
  \includegraphics[width=0.35\textwidth]{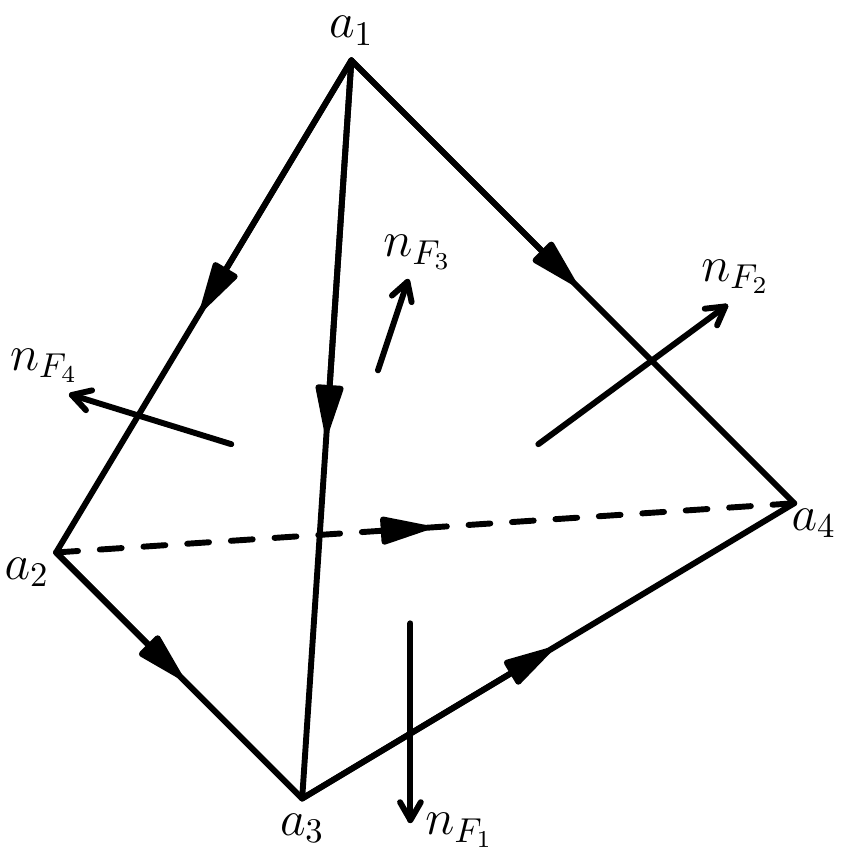}
  \caption{3D Tetrahedron}
  \label{fig:3D-simplex}
\end{figure}
\begin{enumerate}
\item $E$ and $E'$ are same: $w_h = v_h = \varphi_{E}$. Without loss
of generality, we prove the case $E = \overrightarrow{a_1a_2}$, which
follows from 
$$ 
\text{R.H.S. of \eqref{equ:curl-graph}} = 2\omega_{F_3F_4}
\delta_{F_3}(\varphi_{12})\delta_{F_4}(\varphi_{12}) =
-2\omega_{F_3F_4} = \|\curl \varphi_{12}\|_T^2.
$$ 
Here, we use the following formula (cf. \cite[Section
2.1.1]{boffi2013mixed})
$$ 
\delta_F(w_h) = \int_F \curl w_h \cdot n_F = \int_F \div_F(w_h \times
n_F).
$$ 
\item $\bar{E}$ and $\bar{E'}$ share a common vertex. Without loss of
generality, we consider the case in which $E =
\overrightarrow{a_1a_2}, E' = \overrightarrow{a_1a_3}$. Then, 
$$ 
\begin{aligned}
\text{R.H.S. of \eqref{equ:curl-graph}} &= \omega_{F_4F_2}
\delta_{F_4}(\varphi_{12})\delta_{F_2}(\varphi_{13}) \\
&~~~ + \omega_{F_3F_4}\delta_{F_3}(\varphi_{12})
\delta_{F_4}(\varphi_{13}) 
+ \omega_{F_3F_2}\delta_{F_3}(\varphi_{12})\delta_{F_3}(\varphi_{13})
  \\
&= \omega_{F_4F_2} + \omega_{F_3F_4} - \omega_{F_3F_2} \\ 
&= \frac{-\|\curl \varphi_{13}\|_T^2 - \|\curl \varphi_{12}\|_T^2 +
  \|\curl \varphi_{14}\|_T^2}{2} \\
&= -2\|\nabla\lambda_1\times \nabla\lambda_3\|_T^2 
-2\|\nabla\lambda_1\times \nabla\lambda_2\|_T^2 
+2\|\nabla\lambda_1\times \nabla\lambda_4\|_T^2 \\
&= 4(\nabla\lambda_1\times\nabla\lambda_2,
\nabla\lambda_1\times\nabla\lambda_3)_T = (\curl\varphi_{12},
  \curl\varphi_{13})_T.
\end{aligned}
$$ 
\item $\bar{E}\cap \bar{E'} = \varnothing$. Without loss of generality,
we consider the case in which $E = \overrightarrow{a_1a_2}, E' =
\overrightarrow{a_3a_4}$. Then, 
$$ 
\begin{aligned}
\text{R.H.S. of \eqref{equ:curl-graph}} &= \omega_{F_3F_1}
\delta_{F_3}(\varphi_{12})\delta_{F_1}(\varphi_{34}) 
+ \omega_{F_3F_2} \delta_{F_3}(\varphi_{12})\delta_{F_2}(\varphi_{34})\\
&~~~ + \omega_{F_4F_1}\delta_{F_4}(\varphi_{12})
\delta_{F_1}(\varphi_{34}) 
+ \omega_{F_4F_2}\delta_{F_4}(\varphi_{12}) 
\delta_{F_2}(\varphi_{34}) \\
&= \omega_{F_3F_1} - \omega_{F_3F_2} - \omega_{F_4F_1} + \omega_{F_4F_2} \\ 
&= 2\|\nabla\lambda_2\times \nabla\lambda_3\|_T^2 +
2\|\nabla\lambda_1\times \nabla \lambda_4\|_T^2 \\ 
& ~~~ -2\|\nabla\lambda_1\times \nabla\lambda_3\|_T^2 
-2\|\nabla\lambda_2\times \nabla\lambda_4\|_T^2 \\
&= 4(\nabla\lambda_1\times\nabla\lambda_2,
\nabla\lambda_3\times\nabla\lambda_4)_T = (\curl\varphi_{12},
\curl\varphi_{34})_T.
\end{aligned}
$$ 
\end{enumerate}
This completes the proof.
\end{proof}

\begin{lemma} \label{lem:curl-identity}
The following identity holds
\begin{equation} \label{equ:curl-identity}
I = \sum_{F,F'\in \mathcal{S}_T^2, F\neq F'} \omega_{FF'}^T
\frac{|F||F'|}{|T|} n_Fn_{F'}^T = 
\sum_{F,F'\in \mathcal{S}_T^2, F\neq F'} \omega_{FF'}^T
\frac{|F||F'|}{|T|} n_{F'}n_{F}^T.
\end{equation}
\end{lemma}
\begin{proof}
\commentone{The proof follows by taking $w_h =
\frac{1}{2}\xi\times x$ and $v_h = \frac{1}{2}\eta\times x$ in
  \eqref{equ:curl-graph} for arbitrary $\xi, \eta \in \mathbb{R}^3$.}
%%For any $\xi, \eta \in \mathbb{R}^3$, let $u_{\xi} =
%%\frac{1}{2}\xi\times x$ and $\frac{1}{2}\eta\times x$. Then,
%%from \eqref{equ:curl-graph}, we have 
%%$$ 
%%\begin{aligned}
%%|T| \xi^T \eta &= \sum_{F,F'\in \mathcal{S}_T^2, F\neq F'}
%%\omega_{FF'}^T \delta_F(u_\xi)\delta_{F'}(u_\eta) \\ 
%%& = \sum_{F,F'\in \mathcal{S}_T^2, F\neq F'}
%%\omega_{FF'}^T |F||F'| (\xi\cdot n_F)(\eta\cdot n_{F'}) \\
%%&= \xi^T \left( 
%%\sum_{F,F'\in \mathcal{S}_T^2, F\neq F'} \omega_{FF'}^T
%%|F||F'|  n_Fn_{F'}^T \right) \eta.
%%\end{aligned}
%%$$
%%Switching $F$ and $F'$ leads to the second equality in
%%\eqref{equ:curl-identity}.
\end{proof}
\begin{definition}
$\bar{\Pi}_T^2: \mathcal{P}_1^-\Lambda^2(T) \mapsto \mathbb{R}^3$ is
defined by 
\begin{equation} \label{equ:barPi-k2}
\bar{\Pi}_T^2 w_h := \sum_{F,F'\in \mathcal{S}_T^2, F\neq F'}
\omega_{FF'}^T \frac{|F'|}{|T|}n_{F'} l_{F}(w_h) \qquad
\forall w_h \in \mathcal{P}_1^-\Lambda^2(T).
\end{equation}
\end{definition}

\begin{lemma}
If $w_h$ is a constant vector on $T$, then $\bar{\Pi}_T^2 w_h = w_h$. 
\end{lemma}
\begin{proof}
If $w_h$ is constant, then $l_{F}(w_h) = |F|w_h \cdot n_{F}$. Thus, 
$$ 
\begin{aligned}
\bar{\Pi}_T^2 w_h &= \sum_{F,F'\in \mathcal{S}_T^2, F\neq F'}
\omega_{FF'}^T \frac{|F'|}{|T|}|F|(w_h \cdot n_F) n_F' \\ 
&= \left( \sum_{F,F'\in \mathcal{S}_T^2, F\neq F'} \omega_{FF'}^T
\frac{|F||F'|}{|T|}n_{F'} n_{F}^T \right) w_h = w_h.
\end{aligned}
$$ 
This completes the proof.
\end{proof}

By analogy the local SAFE bilinear form for the
$H(\curl)$ convection-diffusion is given as 
\begin{equation} \label{equ:aT-k1} 
a_T(w_h, v_h) := (\alpha\bar{\Pi}_T^2 J_{\bar{\theta},T}^1w_h,
\curl v_h)_T \qquad \forall w_h, v_h \in \mathcal{P}_1^-\Lambda^1(T).
\end{equation} 

\begin{theorem}
It holds that 
\begin{equation} \label{equ:aT2-k1}
a_T(w_h, v_h) = \sum_{F,F'\in \mathcal{S}_T^2, F\neq F'}
\omega_{FF'}^T \mathcal{H}_F(\bar{\alpha}, \bar{\theta})
\delta_F(\exp(\bar{\theta}\cdot x)w_h)\delta_{F'}(v_h).
\end{equation}
\end{theorem}
\begin{proof}
From \eqref{equ:barPi-k2} and Theorem \ref{thm:bT2}, we have  
$$
\begin{aligned}
\bar{\Pi}_T^2 J_{\bar{\theta},T}^1w_h &=
\bar{\Pi}_T^2 \sum_{F\in \mathcal{S}_T^2} \left( \dashint_F
\exp(\bar{\theta} \cdot x) \right)^{-1}
\delta_F(E_{\bar{\theta}}w_h)\varphi_F \\
&= \sum_{F,F'\in \mathcal{S}_T^2, F\neq F'} \left( \dashint_F
\exp(\bar{\theta}\cdot x) \right)^{-1}
\delta_F(E_{\bar{\theta}}w_h) \omega_{FF'}^T
\frac{|F'|}{|T|}n_{F'}.
\end{aligned}
$$ 
Therefore, 
$$ 
\begin{aligned}
a_T(w_h, v_h) &= \sum_{F,F'\in \mathcal{S}_T^2, F\neq F'} \left(
\dashint_F \exp(\bar{\theta}\cdot x)\right)^{-1}
\delta_F(E_{\bar{\theta}}w_h)
\omega_{FF'}^T (\alpha \frac{|F'|}{|T|}n_{F'}, \curl v_h)_T \\
&= \sum_{F,F'\in \mathcal{S}_T^2, F\neq F'}
\omega_{FF'}^T \mathcal{H}_F(\bar{\alpha}, \bar{\theta})
\delta_F(\exp(\bar{\theta} \cdot x)w_h)\delta_{F'}(v_h).
\end{aligned}
$$ 
This completes the proof.
\end{proof}

\subsection{Local bilinear form on $\mathcal{P}_1^{-}\Lambda^2(T)$}
For the $H(\div)$ convection-diffusion problem, since
$\mathcal{P}_0^-\Lambda^3(T)$ is constant, then the operator
$\bar{\Pi}_T^3$ is an identity operator. As a consequence, 
\eqref{equ:bT} can be recast into  
\begin{equation} \label{equ:aT2-k2} 
a_{T}(w_h, v_h) = \omega_T 
\mathcal{H}_{T}(\bar{\alpha}, \bar{\theta})
\delta_T\left( \exp( \bar{\theta} \cdot x) w_h \right) \delta_T(v_h),
\end{equation} 
where $\omega_T = 1/|T|$.

\subsection{Summary of local bilinear forms}
We summarize the operators defined above in \eqref{equ:diagram2}. Note
that the diagrams are commutative when $\theta$ is constant.
\begin{equation} \label{equ:diagram2}
\begin{tikzcd}
\Lambda^0(T) \arrow[r, "J_{\theta}"] 
\arrow[d, swap, "\tilde{\Pi}_{\theta,T}^0"] 
& \Lambda^1(T) \arrow[r, "J_{\theta}"] 
\arrow[d, swap, "\tilde{\Pi}_{\theta,T}^1"] 
& \Lambda^2(T) \arrow[r,"J_{\theta}"] 
\arrow[d, swap, "\tilde{\Pi}_{\theta,T}^2"] 
& \Lambda^3(T) 
\arrow[d, swap, "\tilde{\Pi}_{\theta,T}^3"] \\ 
\mathcal{P}_1^-\Lambda^0(T) \arrow[r, "J_{\theta,T}"] 
& \mathcal{P}_1^-\Lambda^1(T) \arrow[r, "J_{\theta,T}"] 
\arrow[d, swap, "\bar{\Pi}_T^1"] 
& \mathcal{P}_1^-\Lambda^2(T) \arrow[r, "J_{\theta,T}"] 
\arrow[d, swap, "\bar{\Pi}_T^2"] 
& \mathcal{P}_0\Lambda^3(T) 
\arrow[d, swap, "\bar{\Pi}_T^3"] \\
& \mathbb{R}^3 
& \mathbb{R}^3 
& \mathbb{R}
\end{tikzcd}
\end{equation}

The local SAFE bilinear forms for $H(\grad)$, $H(\curl)$, and
$H(\div)$ convection-diffusion problems can be written in a unified
fashion: 
\begin{equation} \label{equ:general-aT}
a_T(w_h, v_h) = (\alpha\bar{\Pi}_T^{k+1} J_{\bar{\theta}, T}^k w_h, d^k
v_h)_T \qquad \forall w_h, v_h \in \mathcal{P}_1^-\Lambda^k(T),
\end{equation}
where $J_{\bar{\theta}, T}^k$ is given in Definition \ref{def:J-T}.
The equivalent forms for $k=0,1,2$, which are suitable for the
implementation, are given in \eqref{equ:aT2-k0}, \eqref{equ:aT2-k1}
and \eqref{equ:aT2-k2}, respectively. The implementation hinges on the
stable discretization of Bernoulli functions, see Appendix
\ref{subsec:Bernoulli}. In addition, the SAFE schemes are shown to
have limiting schemes for vanishing diffusion coefficient $\alpha$, \commentone{
which result in a family of upwind schemes according to the limit of Bernoulli functions},
see Appendix \ref{subsec:limiting}.

%% Global 
\subsection{SAFE schemes}
Let $V_h = \{ v_h \in \mathcal{P}_1^-\Lambda^k(\mathcal{T}_h):~
\tr(v_h) = 0~\mbox{on }\Gamma_0\}$.  Having the local SAFE bilinear
forms \eqref{equ:general-aT}, the global bilinear forms are then
obtained by summing over all the local forms and adding the low-order
terms, i.e.,
\begin{equation} \label{equ:bilinear-a}
a_h(w_h, v_h) = \sum_{T\in \mathcal{T}_h} a_{T}(w_h, v_h) + (\gamma
w_h, v_h). 
\end{equation}
Finally, the finite element approximations of the problems
\eqref{equ:conv-diff} read: Find $u_h \in V_h$ such that 
\begin{equation} \label{equ:SAFE}
a_{h}(u_h, v_h) = F(v_h) \qquad \forall v_h \in V_h.
\end{equation}

For the discretization of dual problems \eqref{equ:conv-diff-dual}, we
simply define $a_h^*(w_h, v_h) = a_h(v_h, w_h)$.  Then, the finite
element approximations of the problems \eqref{equ:conv-diff-dual}
read: Find $u_h \in V_h$ such that 
\begin{equation} \label{equ:SAFE-dual}
a_h^*(u_h, v_h) = F(v_h) \qquad \forall v_h \in V_h.
\end{equation}

\begin{remark}
In \cite[Section 5]{xu1999monotone}, the monotonicity of EAFE requires the
mass-lumping for the low-order term.
\end{remark}

%%%% Error %%%% 
\section{Analysis of Discrete Problems} \label{sec:analysis}
In this section, we analyse the SAFE schemes for the $H(D)$
convection-diffusion problems. As an essential tool, we first
present some local error estimates. Under the well-posedness  of
the model problems, we then establish the well-posedness for the
discrete problems.

\subsection{Local error estimates}
For simplicity, we denote $\bar{\Pi}_{\theta,T}^k = \bar{\Pi}_T^k
\tilde{\Pi}_{\theta,T}^k$. 

%%%%%%%%%%%%%%%%%%%%%%%%%%%%%%%%%%%%%%%%%%%%%%%%%% 
%% Projection %% 
%%%%%%%%%%%%%%%%%%%%%%%%%%%%%%%%%%%%%%%%%%%%%%%%%% 
\begin{lemma} \label{lem:T-error}
For any $T\in \mathcal{T}_h$, if $g\in W^{1,p}(T)$ and $p>n$, we have 
\begin{equation} \label{equ:T-error}
\|g- \tilde{\Pi}_{\theta,T}^k g\|_{0,s,T} \lesssim
C(p)h_T^{1+n({1\over s}-{1\over p})}|g|_{1,p,T} \qquad 1\leq s \leq
\infty.
\end{equation}
Here, $C(p) \eqsim \max \{1, (p-n)^{-\sigma}\}$ where $\sigma$ is a
positive number determined by Sobolev embedding. In addition,
\eqref{equ:T-error} also holds when replacing
$\tilde{\Pi}_{\theta,T}^k$ by  $\bar{\Pi}_{\theta,T}^k$.
\end{lemma}

\begin{proof}
Consider a change of variable from the standard reference element
$\hat{T}$ to $T$: $x = \mathcal{F}(\hat{x}) = B\hat{x} + b_0$. From
the definition of $\tilde{\Pi}_{\theta,T}^k$ in
\eqref{equ:tilde-proj}, the corresponding projection can be written as 
$$ 
\hat{\tilde{\Pi}}_{\hat{\theta}, \hat{T}}^k \hat{g}
= \sum_{\hat{S} \in \mathcal{S}_{\hat{T}}^k}
\frac{l_{\hat{S}}^k(\exp(\hat{\theta}\cdot
\mathcal{F}(\hat{x})) \hat{g})}{\dashint_{\hat{S}} \exp(\hat{\theta} \cdot
\mathcal{F}(\hat{x}))}
\varphi_{\hat{S}}, \qquad \text{where } \hat{\theta}(\hat{x}) =
\theta(\mathcal{F}(\hat{x})). 
$$ 
Since the coefficient of $\varphi_{\hat{S}}$ is a weighted average, we
have 
$$ 
\|\hat{\tilde{\Pi}}_{\hat{\theta}, \hat{T}}^k \hat{g}\|_{0,s,\hat{T}}
\lesssim \|\hat{g}\|_{0,\infty,\hat{T}},
$$ 
where the hidden constant does not depend on $\theta$. By the Sobolev
embedding theorem (cf. \cite{adams2003sobolev}), $W^{1,p}(\hat{T})
\hookrightarrow L^{\infty}(\hat{T})$ when $p>n$, we get  
$$ 
\|\hat{\tilde{\Pi}}_{\hat{\theta}, \hat{T}}^k \hat{g}\|_{0,s,\hat{T}}
\lesssim \|\hat{g}\|_{0,\infty,\hat{T}} \lesssim C(p)
\|\hat{g}\|_{1,p,\hat{T}}.
$$ 

From the definition of the interpolation operator,
$\tilde{\Pi}_{\theta,T}^k g = g$ (or $\bar{\Pi}_{\theta,T}^kg = g$) if
$g$ is constant on $T$.  By the Bramble-Hilbert lemma and scaling
argument (see \cite[Section 2.1.3]{boffi2013mixed} for Piola
transformation for $H(\curl)$ and $H(\div)$ spaces), we have
$$
\|g-\tilde{\Pi}_{\theta,T}^{k} g\|_{0,s,T}\lesssim h_T^{n\over s}
\|\hat g-\hat{\tilde \Pi}_{\hat{\theta}, \hat T}^{k}\hat g\|_{0,s,\hat T}
\lesssim C(p) h_T^{n\over s}|\hat g|_{1, p,\hat T} \lesssim
C(p)h_T^{1 + n(\frac{1}{s} - \frac{1}{p})} |g|_{1,p,T}.
$$
The estimate for $\bar{\Pi}_{\theta,T}^k$ follows from a similar
argument. 
\end{proof}
In the proof of Lemma \ref{lem:T-error}, we have the following
stability of $\tilde{\Pi}_{\theta,T}$.
\begin{corollary} \label{cor:stab-Pi}
For any $T\in \mathcal{T}_h$, if $w \in L^\infty(T)$, we have 
\begin{equation} \label{equ:stab-Pi}
\|\tilde{\Pi}_{\theta,T}w\|_{0,s,T} \lesssim h_T^{n\over
s}\|w\|_{0,\infty,T} \qquad 1 \leq s \leq \infty,
\end{equation}
where the hidden constant does not depend on $\theta$.
\end{corollary}

We now want to analyse the behavior of the
$\tilde{\Pi}_{\bar{\theta},T}^{k+1}J_{\bar{\theta}}^k w -
J_{\bar{\theta},T}^k \tilde{\Pi}_{\theta,T}^k w$. According to the
commutative diagram \eqref{equ:diagram2}, we deduce that 
$$ 
\tilde{\Pi}_{\bar{\theta},T}^{k+1}J_{\bar{\theta}}^k w -
J_{\bar{\theta},T}^k \tilde{\Pi}_{\theta,T}^k w 
%=  J_{\bar{\theta},T}^k \tilde{\Pi}_{\bar{\theta},T}^k(w) -
%J_{\bar{\theta},T}^k \tilde{\Pi}_{\theta,T}^k(w)
= J_{\bar{\theta},T}^k (\tilde{\Pi}_{\bar{\theta},T}^k w - 
\tilde{\Pi}_{\theta,T}^k w).
$$
Let $x_c$ be the barycenter of $T$. The main observation is that
$\tilde{\Pi}_{\theta,T}^k$ (resp. $\tilde{\Pi}_{\bar{\theta},T}$) does
not change under the transformation $\theta \cdot x \mapsto \theta
\cdot x - \bar{\theta}\cdot x_c$ (resp. $\bar{\theta} \cdot x \mapsto
\bar{\theta} \cdot x - \bar{\theta}\cdot x_c$ ).

\begin{lemma} \label{lem:commute-diff}
For any $T\in \mathcal{T}_h$, if $w\in W^{1,p}(T)$, $p>n$, and $h_T
\lesssim \|\theta\|_{1,\infty,T}^{-1}$, we have 
$$ 
\|(\tilde{\Pi}_{\theta,T}^k -
\tilde{\Pi}_{\bar{\theta},T}^k)w\|_{0,s,T} \lesssim C(p)h_T^{2 +
n({1\over s} - {1\over p})} |\theta|_{1,\infty,T}
|w|_{1,p,T} \quad 1\leq s \leq \infty.
$$ 
\end{lemma}
\begin{proof} 
In light of the definition of $\tilde{\Pi}_{\theta,T}^k$ in
\eqref{equ:tilde-proj}, dividing $\exp(-\bar{\theta}\cdot x_c)$ on
both numerator and denominator of the coefficient of $\varphi_S$, we
have 
$$ 
\tilde{\Pi}_{\theta,T}^k v 
= \sum_{S \in \mathcal{S}_T^k} \frac{l_S^k(\exp(\theta\cdot
x)v)}{\dashint_S \exp(\theta\cdot x)} \varphi_S
= \sum_{S \in \mathcal{S}_T^k} \frac{l_S^k(\exp(\theta\cdot
x - \bar{\theta}\cdot x_c)v)}{\dashint_S \exp(\theta\cdot x -
\bar{\theta}\cdot x_c)} \varphi_S.
$$ 
Then, for any $x \in S$, we have $\|\theta\cdot x - \bar{\theta}\cdot
x_c\|_{0,\infty,T} \lesssim h_T\|\theta\|_{1,\infty,T} \lesssim 1$ and
therefore 
$$ 
\begin{aligned}
|\exp(\theta \cdot x - \bar{\theta}\cdot x_c) - \exp(\bar{\theta}\cdot
x - \bar{\theta}\cdot x_c)| &= 
\exp(\bar{\theta}\cdot x - \bar{\theta}\cdot x_c)|1 - \exp((\theta -
\bar{\theta}) \cdot x)| \\
& \lesssim 
h_T |\theta|_{1,\infty,T}.
\end{aligned}
$$ 
Then, we have the estimates of the numerator and denominator 
$$ 
\begin{aligned}
|\dashint_S \exp(\theta\cdot x - \bar{\theta}\cdot x_c)v -
\exp(\bar{\theta}\cdot x - \bar{\theta}\cdot x_c) v| &\lesssim 
h_T |\theta|_{1,\infty,T} \frac{\|v\|_{0,1,S}}{|S|} \\
&\lesssim 
h_T |\theta|_{1,\infty,T} \|v\|_{0,\infty,S}, \\ 
1 \lesssim \dashint_S \exp(\theta\cdot x - \bar{\theta}\cdot x_c) \leq
\dashint_S &\exp(\bar{\theta}\cdot x - \bar{\theta}\cdot x_c) + Ch_T
|\theta|_{1,\infty,T}.
\end{aligned}
$$ 
Therefore, 
$$ 
\left|
\frac{l_S^k(\exp(\theta\cdot x)v)}{\dashint_S \exp(\theta\cdot x)} - 
\frac{l_S^k(\exp(\bar{\theta}\cdot x)v)}{\dashint_S \exp(\bar{\theta}
\cdot x)}\right| 
\lesssim h_T |\theta|_{1,\infty,T}\|v\|_{0,\infty,S}. 
$$ 
Note that, for any $w_h \in \mathcal{P}_1^-\Lambda^k(T)$, $
(\tilde{\Pi}_{\theta,T}^k - \tilde{\Pi}_{\bar{\theta},T}^k)w =
(\tilde{\Pi}_{\theta,T}^k - \tilde{\Pi}_{\bar{\theta},T}^k)(w - w_h)$. 
Taking $v = w - w_h$, by the Bramble-Hilbert lemma and the standard
scaling argument, we obtain the desired result.
\end{proof}

%% Error analysis %%
\subsection{Error Analysis}
Define the special interpolations $\tilde{\Pi}_{\theta,h}^k$ by $
\tilde{\Pi}_{\theta,h}^k w|_T := \tilde{\Pi}_{\theta,T}^k w$ for any
$T\in \mathcal{T}_h$.  \commentone{In light of local error estimates}, we first give an estimate for the difference
between continuous and approximating bilinear forms. \commentone{Note that the solution of convection-diffusion problems may have boundary or internal layer, the analysis in this section hinges on the assumption that $h$ is sufficiently small.} 

%% difference
\begin{lemma} \label{lem:difference}
For any $T\in \mathcal{T}_h$, assume that $h_T
\lesssim \|\theta\|_{1,\infty,T}^{-1}$, 
$J_{\bar{\theta}}^k w = d^k u + i_{\bar{\theta}}^*u \in W^{1,p}(T)$ and $w
\in W^{1,r}(T)$ where $p, r > n$. Then, the following inequality holds 
\begin{equation}
|a(w, v_h) - a_h(\tilde{\Pi}_{\theta,h}^kw, v_h)| \lesssim \Theta_1(\alpha,
\theta, \gamma, w)h \|v_h\|_{H\Lambda,\Omega} \qquad \forall v_h \in V_h,
\end{equation}
where 
\begin{equation} \label{equ:C1} 
\begin{aligned}
\Theta_1&(\alpha,\theta,\gamma,w):= \\ 
\bigg\{ &
\sum_{T\in \mathcal{T}_h } \left(
\|\alpha\|_{0,\infty,T}|\theta|_{1,\infty,T}\|w\|_{0,T} \right)^2
\\ + & 
\sum_{T\in \mathcal{T}_h } 
 \left( \|\alpha\|_{0,\infty,T}C(p)h_T^{n({1\over2}- {1\over p})} 
|J_{\bar{\theta}} w|_{1,p,T} \right)^2 \\
+ &
\sum_{T\in \mathcal{T}_h } 
\left( \|\alpha\|_{0,\infty,T} |\theta|_{1,\infty,T}(1 +
h_T\|\theta\|_{0,\infty,T}) C(r)h_T^{n({1\over2}- {1\over r})}
|w|_{1,r,T} \right)^2 \\
+ & 
\sum_{T\in \mathcal{T}_h}\left(
\|\gamma\|_{0,\infty,T}C(r)h_T^{n({1\over2}- {1\over r})} |w|_{1,r,T}
\right)^2 \bigg\}^{1\over 2}.
\end{aligned}
\end{equation}
\end{lemma}
\begin{proof}
By \eqref{equ:general-aT} and the
diagram \eqref{equ:diagram2}, we have 
$$ 
\begin{aligned}
a(w, v_h) - a_h(\tilde{\Pi}_{\theta,h}^k w, v_h) &=
\sum_{T\in \mathcal{T}_h}(\alpha J_{\theta}^k w - \alpha
\bar{\Pi}_T^{k+1} J_{\bar{\theta},T}^k \tilde{\Pi}_{\theta,T}^kw, d^k
v_h)_T \\ 
&~ + \sum_{T\in \mathcal{T}_h} (\gamma (w-\tilde{\Pi}_{\theta,T}^kw),
    v_h)_T
\\&= \sum_{T\in \mathcal{T}_h} \underbrace{ 
(\alpha i_{\theta - \bar{\theta}}^*w, d^k v_h)_T
}_{I_{1,T}} + 
\underbrace{ ( \alpha (I - \bar{\Pi}_{\bar{\theta},T}^{k+1})
J_{\bar{\theta}}^k w , d^k v_h)_T }_{I_{2,T}} 
\\ 
&~ + \sum_{T\in \mathcal{T}_h} \underbrace{(\alpha \bar{\Pi}_T^{k+1}
J_{\bar{\theta},T}^k(\tilde{\Pi}_{\bar{\theta},T}^k -
\tilde{\Pi}_{\theta,T}^k) w, d^k v_h)_T}_{I_{3,T}}
\\
&~ + \sum_{T\in \mathcal{T}_h} \underbrace{(\gamma
    (w-\tilde{\Pi}_{\theta,T}^kw), v_h)_T}_{I_{4,T}}.
\end{aligned}
$$ 
Clearly, 
\begin{equation} \label{equ:I1T} 
|I_{1,T}| \lesssim h_T\|\alpha\|_{0,\infty,T} |\theta|_{1,\infty,T}
\|w\|_{0,T} \|d^kv_h\|_{0,T}. 
\end{equation} 
Thanks to Lemma \ref{lem:T-error}, we have 
\begin{align} 
|I_{2,T}| &\leq \|\alpha\|_{0,\infty,T} C(p) h_T^{1 + n({1\over 2} -
    {1\over p})} |J_{\bar{\theta}}^k w|_{1,p,T} \|d^k v_h\|_{0,T},
\label{equ:I2T} \\
|I_{4,T}| &\leq \|\gamma\|_{0,\infty,T} C(r) h_T^{1 + n({1\over 2} -
    {1\over r})} |w|_{1,r,T} \|v_h\|_{0,T}.
\label{equ:I4T} 
\end{align}
Using inverse inequality, Corollary \ref{cor:stab-Pi} and Lemma
\ref{lem:commute-diff}, we have 
\begin{equation} \label{equ:I3T}
\begin{aligned}
|I_{3,T}| & \lesssim \|\alpha\|_{0,\infty,T} \|J_{\bar{\theta},T}^k
(\tilde{\Pi}_{\bar{\theta},T}^k - \tilde{\Pi}_{\theta,T}^k) w \|_{0,T}
\|d^k v_h\|_{0,T} \\
& = \|\alpha\|_{0,\infty,T}
\|\tilde{\Pi}_{\bar{\theta},T}^{k+1}J_{\bar{\theta}}^k
(\tilde{\Pi}_{\bar{\theta},T}^k - \tilde{\Pi}_{\theta,T}^k) w \|_{0,T}
\|d^k v_h\|_{0,T} \\
& \lesssim \|\alpha\|_{0,\infty,T}  h_T^{n\over 2}\|J_{\bar{\theta}}^k
(\tilde{\Pi}_{\bar{\theta},T}^k - \tilde{\Pi}_{\theta,T}^k) w
\|_{0,\infty,T} \|d^k v_h\|_{0,T} \\
& \lesssim \|\alpha\|_{0,\infty,T} h_T^{n\over 2} 
\big( \|d (\tilde{\Pi}_{\bar{\theta},T}^k - \tilde{\Pi}_{\theta,T}^k)
    w\|_{0,\infty,T} \\
& \qquad \qquad \qquad + \|\theta\|_{0,\infty,T} 
\|(\tilde{\Pi}_{\bar{\theta},T}^k - \tilde{\Pi}_{\theta,T}^k)
w\|_{0,\infty,T} \big) \|d^kv_h\|_{0,T} \\
& \lesssim \|\alpha\|_{0,\infty,T} |\theta|_{1,\infty,T} (1 +
h_T\|\theta\|_{0,\infty,T}) C(r)h_T^{1 + n({1\over 2}-{1\over r})}
|w|_{1,r,T} \|d^k v_h\|_{0,T}.
\end{aligned}
\end{equation}
By \eqref{equ:I1T} -- \eqref{equ:I4T}, we obtain the desired results. 
\end{proof}

\begin{remark}
\commentone{
In the above lemma, if the diffusion coefficient $\alpha$ is piecewise
constant, we have $ \|\alpha\|_{0,\infty, T}|\theta|_{1,\infty,T} =
|\beta|_{1,\infty,T}$, which describes the variation rate of
convection speed in element $T$. 
}
\end{remark}

%% well-posedness
\begin{theorem} \label{thm:inf-sup-h}
Under the Assumption \ref{asp}, for sufficiently small $h$, both
\eqref{equ:SAFE} and \eqref{equ:SAFE-dual} are well-posed and
furthermore the following inf-sup conditions hold:
\begin{equation} \label{equ:inf-sup-h}
\inf_{w_h\in V_h}\sup_{v_h\in V_h} 
\frac{a_h(w_h,v_h)}{\|w_h\|_{H\Lambda,\Omega}\|v_h\|_{H\Lambda,\Omega}} =
\inf_{w_h\in V_h}\sup_{v_h\in V_h}
\frac{a_h^*(w_h,v_h)}{\|w_h\|_{H\Lambda,\Omega}\|v_h\|_{H\Lambda,\Omega}}
= c_1>0.
\end{equation}
\end{theorem}

\begin{proof}
It is well-known (c.f. Schatz \cite{schatz1974observation}, Xu
\cite{xu1996two}) that, thanks to \eqref{equ:inf-sup}, the bilinear form
$a(u_h,v_h)$ satisfies discrete inf-sup condition as for sufficiently
small $h$:
$$ 
\sup_{v_h\in V_h} \frac{a(w_h, v_h)}{\|v_h\|_{H\Lambda,\Omega}} \geq
\frac{c_0}{2} \|w_h\|_{H\Lambda, \Omega}, \quad 
\sup_{v_h\in V_h} \frac{a^*(w_h, v_h)}{\|v_h\|_{H\Lambda,\Omega}} \geq
\frac{c_0}{2} \|w_h\|_{H\Lambda, \Omega} \qquad w_h \in V_h. 
$$
It follows from Lemma \ref{lem:difference} that  
$$
|a(w_h,v_h)- a_h(w_h,v_h)| \lesssim \Theta_1(\alpha,\theta,\gamma,w_h)
h \|v_h\|_{H\Lambda,\Omega}.
$$
Observe that $|d^k w_h|_{1,p,T} = 0$ for any $w_h \in V_h$ and $T\in
\mathcal{T}_h$. By inverse equality, we have the estimate of discrete
flux 
$$ 
h_T^{n({1\over2} - {1\over p})}|J^k_{\bar{\theta}} w_h|_{1,p,T} = 
h_T^{n({1\over2} - {1\over p})}\|\bar{\theta}\|_{0,\infty,T}
|w_h|_{1,p,T} \lesssim \|\theta\|_{0,\infty,T}
\|w_h\|_{H\Lambda,T}.
$$ 
The rest of the terms in $\Theta_1(\alpha, \theta, \gamma, w_h)$ can be
estimated by the inverse inequality. That is,   
\begin{equation} \label{equ:discrete-dual} 
|a(w_h, v_h) - a_h(w_h, v_h)| \lesssim \Theta_2(\alpha, \theta,
\gamma) h \|w_h\|_{H\Lambda,\Omega} \|v_h\|_{H\Lambda,\Omega},
\end{equation} 
where 
\begin{equation} \label{equ:C2} 
\begin{aligned}
\Theta_2(\alpha,\theta,\gamma) := 
\max_{T\in \mathcal{T}_h } 
\bigg\{ &
 \left( C(p) \|\alpha\|_{0,\infty,T}\|\theta\|_{0,\infty,T}\right)^2 
+ \left( C(r) \|\gamma\|_{0,\infty,T} \right)^2 \\
+ & \left(  C(r)\|\alpha\|_{0,\infty,T} |\theta|_{1,\infty,T}(1 +
h_T\|\theta\|_{0,\infty,T}) \right)^2 \bigg\}^{1\over 2}.
\end{aligned}
\end{equation} 
The desired result then follows when 
\begin{equation} \label{equ:inf-sup-h-constrain} 
h \lesssim h_0 := c_0\min\left\{ 
\|\theta\|_{1,\infty}^{-1}, \Theta_2(\alpha,\theta,\gamma)^{-1}
\right\}.
\end{equation} 
\end{proof}

%% error estimate 
We have the following convergence results for problems
\eqref{equ:conv-diff} and \eqref{equ:conv-diff-dual}.
\begin{theorem} \label{thm:convergence}
Let $u$ be the solution of the problem \eqref{equ:conv-diff}. Assume
that for all $T \in \mathcal{T}_h$, $u \in W^{1,r}(T)$ and
$J^k_{\bar{\theta}}u \in W^{1,p}(T)$, $p,r>n$. Then, the following
estimate holds for sufficiently small $h$:
\begin{equation} \label{equ:estimate-primal}
\|u_h - \tilde{\Pi}^k_{\theta,h} u\|_{H\Lambda,\Omega} \lesssim
\frac{1}{c_1}\Theta_1(\alpha, \theta, \gamma, u)h.
\end{equation}
\end{theorem}
\begin{proof}
By Lemma \ref{lem:difference},
$$
\begin{aligned}
a_h(u_h- \tilde{\Pi}^k_{\theta,h}u, v_h) 
&= (f, v_h) - a_h(\tilde{\Pi}^k_{\theta,h}u, v_h) = a(u, v_h) -
a_h(\tilde{\Pi}^k_{\theta,h}u, v_h)\\
&\lesssim \Theta_1(\alpha,\theta,\gamma,u)h \|v_h\|_{H\Lambda,\Omega}.
\end{aligned}
$$
By the discrete inf-sup condition \eqref{equ:inf-sup-h},
$$
\|u_h - \tilde{\Pi}_{\theta,h}^ku\|_{H\Lambda,\Omega}\lesssim
\frac{1}{c_1}\Theta_1(\alpha,\theta,\gamma,u)h.
$$
This completes the proof.
\end{proof}

\begin{theorem} \label{thm:convergence-dual}
Let $u$ be the solution of the dual problem
\eqref{equ:conv-diff-dual}. Assume that for all $T \in \mathcal{T}_h$, $h_T
\lesssim \|\theta\|_{1,\infty,T}^{-1}$, 
$u \in W^{1,r}(T)$ and $J^k_{\bar{\theta}}u \in W^{1,p}(T)$, $p,r>n$.
Then the following estimate holds for sufficiently small $h$:
\begin{equation} \label{equ:convergence-dual}
\|u - u_h\|_{H\Lambda,\Omega} \lesssim (1+\frac{M}{c_1})\inf_{w_h \in V_h} \|u -
w_h\|_{H\Lambda,\Omega} + \frac{1}{c_1}\tilde{\Theta}_2(\alpha,\theta,\gamma)
h|\ln h|^\sigma\|u\|_{H\Lambda,\Omega},
\end{equation}
where $M$ is the upper bound of the bilinear form, i.e. $a(u, v) \leq
M\|u\|_{H\Lambda,\Omega}\|v\|_{H\Lambda,\Omega}$, and  
\begin{equation} \label{equ:tilde-Theta2}
\tilde{\Theta}_2(\alpha,\theta,\gamma) := 
\max_{T\in \mathcal{T}_h } 
\bigg\{ 
 \left( \|\alpha\|_{0,\infty,T}\|\theta\|_{1,\infty,T}\right)^2 +
 \left( \|\gamma\|_{0,\infty,T} \right)^2 \bigg\}^{1\over 2}.
\end{equation}
\end{theorem}
\begin{proof}
For sufficiently small $h$, we can take $p = n + |\ln h|^{-1}$ and $r
= n+ |\ln h|^{-1}$.  By the boundedness of bilinear form and
\eqref{equ:discrete-dual},
$$
\begin{aligned}
& \quad a_h^*(u_h-w_h, v_h) \\ 
&= (f, v_h) - a_h^*(w_h, v_h) \\ 
&= a^*(u-w_h, v_h) + a^*(w_h, v_h) - a_h^*(w_h, v_h) \\
&\lesssim M\|u-w_h\|_{H\Lambda,\Omega}\|v_h\|_{H\Lambda,\Omega}
+ \Theta_2(\alpha,\theta, \gamma)h\|w_h\|_{H\Lambda,\Omega}\|
v_h\|_{H\Lambda,\Omega} \\
& \lesssim M\|u-w_h\|_{H\Lambda,\Omega}\|v_h\|_{H\Lambda,\Omega}
+\tilde{\Theta}_2(\alpha,\theta,\gamma)h |\ln
h|^\sigma\|w_h\|_{H\Lambda,\Omega}\| v_h\|_{H\Lambda,\Omega}.
\end{aligned}
$$
Again, by the discrete inf-sup condition \eqref{equ:inf-sup-h}, we
deduce that
$$
\begin{aligned}
\|u_h-w_h\|_{H\Lambda,\Omega} & \lesssim
\frac{M}{c_1}\|u-w_h\|_{H\Lambda,\Omega}
+ \frac{1}{c_1}\tilde{\Theta}_2(\alpha,\theta,\gamma) h|\ln
h|^\sigma\|w_h\|_{H\Lambda,\Omega} \\
& \lesssim \frac{M}{c_1}\|u-w_h\|_{H\Lambda,\Omega} +
\frac{1}{c_1}\tilde{\Theta}_2(\alpha,\theta,\gamma) h |\ln
h|^\sigma\|u\|_{H\Lambda,\Omega}.
\end{aligned}
$$
Thus, by triangle inequality, we obtain the desired result.
\end{proof}

%%%% Numerical %%%%
\section{Numerical Tests} \label{sec:numerical}
In this section, we present several numerical tests in both 2D
and 3D to show the convergence of SAFE scheme as well as the
performance for convection-dominated problems. We set $\bar{\theta}|_T
= \theta(x_c|_T)$ on each element $T$.  The uniform
meshes with different mesh sizes are applied in all the tests.

\subsection{$H({\rm div})$ convection-diffusion in 2D} The 
$\mathcal{P}_1^-\Lambda^k$ discrete de Rham complex in 2D is 
$$
\mathcal{P}_1^-\Lambda^0 \xrightarrow{\curl} \mathcal{P}_1^- \Lambda^1
\xrightarrow{\div} \mathcal{P}_1^-\Lambda^2,
$$
where the 2D $\curl$ operator is defined by $\curl\phi =
(\partial_{x_2}\phi, -\partial_{x_1}\phi)^T$.  Therefore, when $k=1$
in 2D, the operator $\mathcal{L}$ in the boundary value problem
\eqref{equ:conv-diff} can be written as  
$$ 
\mathcal{L}u = -\grad (\alpha \div u + \beta \cdot u) + \gamma u.
$$ 
The computational domain is the square $\Omega = (0,1)^2$, and
$\Gamma_0 = \partial\Omega$. \commentone{That is, the homogeneous
boundary condition $u\cdot n|_{\partial\Omega} = 0$ is applied.} The
convection speed is set to be $\beta = (-x_2, x_1)$.

\paragraph{Convergence order test} $f$ is analytically derived so that
the exact solution of \eqref{equ:conv-diff} is 
$$ 
u = \begin{pmatrix}
e^{x_1 - x_2} x_1x_2(1-x_1)(1-x_2) \\
\sin(\pi x_1)\sin(\pi x_2)
\end{pmatrix}.
$$ 

As shown in Table \ref{tab:div-error1}, the first-order
convergence is observed for both $L^2$ and $H({\rm div})$ errors when
$\alpha = 1, \gamma = 1$. For the case in which $\alpha = 0.01$, no
convergence order is observed for $H({\rm div})$ error when the ratio
$h/\alpha$ is rather large. With the growth of $1/h$, the discrete
system becomes more and more diffusion-dominated. Thus, the first-order
convergence rate in $H({\rm div})$ norm is gradually shown up. To our
surprise, \commentone{for the solution without boundary or internal
layer}, the $L^2$ convergence order of SAFE seems to be stable with
respect to the diffusion coefficient $\alpha$, see Table
\ref{tab:div-error2}.

\begin{table}[!htbp]
\centering 
\captionsetup{justification=centering}
\subfloat[$\alpha=1,\gamma=1$]{
\footnotesize
\begin{tabular}{|l|cc|cc|} 
\hline  
$1/h$ & $\|\epsilon_h\|_0$ & $h^n$ &   
$\|{\rm div}\epsilon_h\|_0$ & $h^n$\\ \hline    
$4$ & 0.151320 & --- &   
0.423821 & --- \\ 
$8$ & 0.077022 & 0.97 &   
0.215003 & 0.98 \\ 
$16$ & 0.038693 & 0.99 &   
0.107889 & 0.99 \\   
$32$ & 0.019370 & 1.00 &   
0.053993 & 1.00 \\ 
$64$ & 0.009688 & 1.00 &   
0.027003 & 1.00 \\ 
$128$ & 0.004844 & 1.00 &   
0.013502 & 1.00 \\ \hline
\end{tabular}
\label{tab:div-error1}
}
\subfloat[$\alpha=0.01, \gamma=1$] {
\footnotesize
\begin{tabular}{|l|cc|cc|}
\hline  
$1/h$ & $\|\epsilon_h\|_0$ & $h^n$ &   
$\|{\rm div}\epsilon_h\|_0$ & $h^n$\\ \hline    
$4$ & 0.169304 & --- &   
0.917169 & --- \\ 
$8$ & 0.084289 & 1.01 &   
0.876446 & 0.07 \\ 
$16$ & 0.040676 & 1.05 &   
0.744944 & 0.23 \\   
$32$ & 0.019737 & 1.04 &   
0.494417 & 0.59 \\ 
$64$ & 0.009741 & 1.02 &   
0.273080 & 0.86 \\ 
$128$ & 0.004851 & 1.01 &   
0.140387 & 0.96 \\ \hline
\end{tabular}
\label{tab:div-error2}
}
\caption{The error, $\epsilon_h = u - u_h$, and convergence order for
2D $H({\rm div})$ convection-diffusion problems.} \label{tab:div-error}
\end{table}

\paragraph{Numerical stability} We set $f = (1, 1)^T$ and $h = 1/32$.
We observe that, when $\alpha = 2\times 10^{-3}$, the SAFE
discretization is stable (Figure \ref{fig:div-SAFE1}), in comparison
with the standard conforming discretization based on the $H(\rm div)$
variational form, which suffers from spurious oscillation (Figure
\ref{fig:div-standard}).  

\begin{figure}[!htbp]
\centering 
\captionsetup{justification=centering}
\subfloat[$\alpha = 2\times 10^{-3}$, standard conforming
discretization]{
  \includegraphics[width=0.45\textwidth]{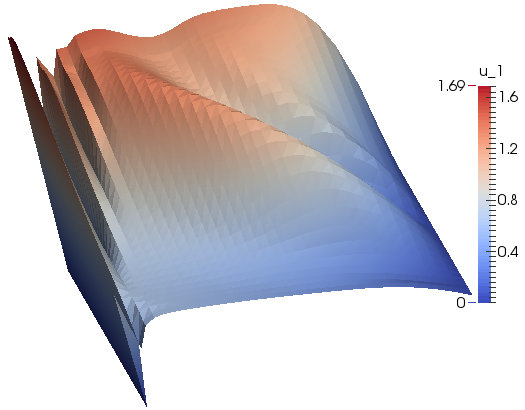}
  \label{fig:div-standard}
}\quad %
\subfloat[$\alpha = 2\times 10^{-3}$, SAFE]{
  \includegraphics[width=0.45\textwidth]{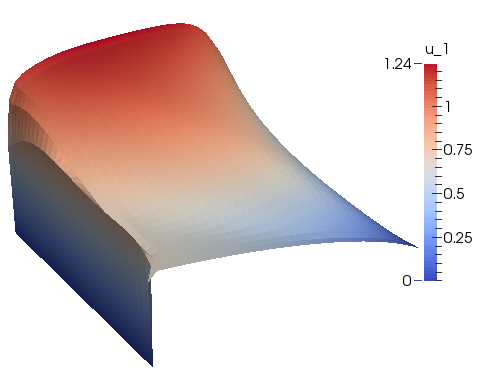}
  \label{fig:div-SAFE1}
} \\
\subfloat[$\alpha = 1\times 10^{-5}$, SAFE]{
  \includegraphics[width=0.45\textwidth]{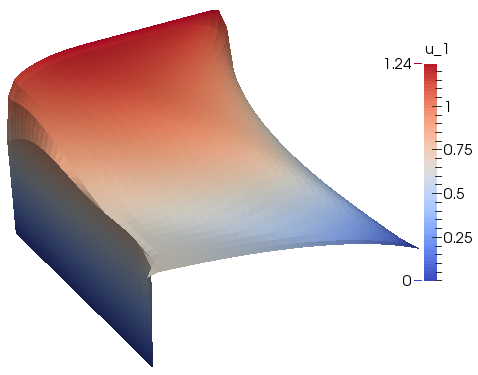}
  \label{fig:div-SAFE2}
} \quad %
\subfloat[$\alpha = 1\times 10^{-7}$, SAFE]{
  \includegraphics[width=0.45\textwidth]{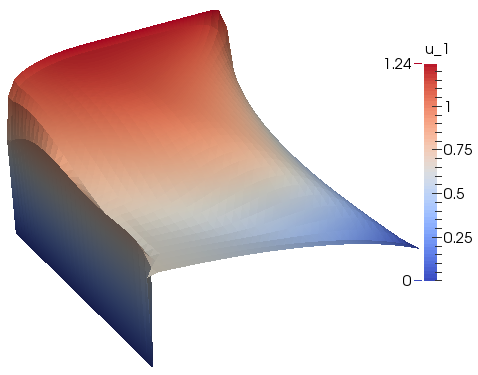}
  \label{fig:div-SAFE3}
}
\caption{Plots of $u_1$ for 2D $H({\rm div})$ convection-diffusion
problems.} 
\end{figure}

Moreover, we take the diffusion coefficient $\alpha = 1\times
10^{-7}$.  Compared to the convection speed $\beta$, the ratio $h /
\alpha = 312500$ is rather large. Fig.
\ref{fig:div-SAFE2}-\ref{fig:div-SAFE3} clearly shows that there is no
spurious oscillation or smearing near the boundary or internal layers
for SAFE. In addition, the numerical solutions under the given mesh
are shown to converge as $\alpha \to 0$, which confirms the results in 
Appendix \ref{subsec:limiting}.

\subsection{$H({\rm curl})$ \commentone{convection-diffusion} in 3D}
The $H(\curl)$ convection-diffusion problem \eqref{eq:E-cd} is
exactly the model problem \eqref{equ:conv-diff-dual} when $k = 1$ in
3D. The numerical test is taken on the unit cube $\Omega = (0,1)^3$.
Let the exact solution be 
$$ 
u = 
\begin{pmatrix}
\sin x_3 \\ \sin x_1 \\ \sin x_2
\end{pmatrix}.
$$ 
Let $\Gamma_0 = \partial \Omega$ and the convection speed be $\beta =
(x_2, x_3, x_1)^T$. The Dirichlet boundary condition and
$f$ can be analytically derived.

\begin{table}[!htbp]
\centering 
\captionsetup{justification=centering}
\subfloat[$\alpha=1,\gamma=1$]{
\footnotesize
\begin{tabular}{|l|cc|cc|} 
\hline  
$1/h$ & $\|\epsilon_h\|_0$ & $h^n$ &   
$\|{\rm curl}\epsilon_h\|_0$ & $h^n$\\ \hline    
$2$ & 0.259495 & --- &   
0.108122 & --- \\ 
$4$ & 0.129934 & 0.99 &   
0.053325 & 1.02 \\ 
$8$ & 0.064987 & 1.00 &   
0.026350 & 1.02 \\ 
$16$ & 0.032496 & 1.00 &   
0.013083 & 1.01 \\   
\hline
\end{tabular}
\label{tab:curl-error1}
}
\subfloat[$\alpha=0.02, \gamma=1$] {
\footnotesize
\begin{tabular}{|l|cc|cc|}
\hline  
$1/h$ & $\|\epsilon_h\|_0$ & $h^n$ &   
$\|{\rm curl}\epsilon_h\|_0$ & $h^n$\\ \hline    
$2$ & 0.267544 & --- &   
0.199419 & --- \\ 
$4$ & 0.151569 & 0.82 &   
0.178998 & 0.16 \\ 
$8$ & 0.089043 & 0.77 &   
0.120931 & 0.57 \\ 
$16$ & 0.047192 & 0.92 &   
0.057896 & 1.06 \\   
\hline
\end{tabular}
\label{tab:curl-error2}
}
\caption{The error, $\epsilon_h = u - u_h$, and convergence order for
3D $H({\rm curl})$ convection-diffusion problems.} \label{tab:curl-error}
\end{table}

As shown in Table \ref{tab:curl-error}, the first-order convergence
is observed for both $L^2$ and $H({\rm curl})$ errors when $\alpha =
1, \gamma = 1$. In addition, when the convection and $h/\alpha$ are of
the same order of magnitude, the first-order convergence for $H({\rm
curl})$ error is observed.  

%%%% Discussions %%%%
\appendix
\section{Implementation issues and limiting case} \label{sec:append}
We shall discuss the implementation of SAFE, and briefly discuss the
limiting case when the diffusion coefficient approaches to zero. 
\subsection{Bernoulli functions} \label{subsec:Bernoulli}
In light of \eqref{equ:aT2-k0}, \eqref{equ:aT2-k1} and
\eqref{equ:aT2-k2}, the local SAFE stiffness matrix is assembled by  
$\omega_E^T$, $\omega_{FF'}^T$ or $\omega_T$, which is determined
by the local stiffness matrix of $(d^kw_h, d^kv_h)_T$ or the geometric
information of $T$, and the
following coefficients:
$$ 
\mbox{diffusion coefficient} \times
\frac{\mbox{exponential average on sub-simplex of dimension }k}
{\mbox{exponential average on sub-simplex of dimension }k+1}.
$$ 
Therefore, thanks to the affine mapping to reference element, the
implementation of the SAFE hinges on the following Bernoulli functions: 
\begin{subequations} \label{equ:Bernoulli} 
\begin{align}
B_1^\epsilon(s) &:= \epsilon \frac{1}{\int_0^1 \exp(s
\hat{x}_1/\epsilon)\, \mathrm{d}\hat{x}_1}, \label{equ:B1}\\ 
B_2^\epsilon(s,t) &:= \epsilon \frac{\int_0^1 \exp(s
\hat{x}_1/\epsilon)\,\mathrm{d}\hat{x}_1}
{2 \int_0^1 \int_0^{1-\hat{x}_2}
\exp((s\hat{x}_1 + t\hat{x}_2)/\epsilon)\,\mathrm{d}\hat{x}_1 
\mathrm{d}\hat{x}_2}, \label{equ:B2}\\
B_3^\epsilon(s,t,r) &:= \epsilon \frac{2 \int_0^1 \int_0^{1-\hat{x}_2}
\exp((s\hat{x}_1 +
t\hat{x}_2)/\epsilon)\,\mathrm{d}\hat{x}_1\mathrm{d}\hat{x}_2}
{6\int_0^1\int_0^{1-\hat{x}_3} \int_0^{1-\hat{x}_2 - \hat{x}_3} 
\exp((s\hat{x}_1 + t\hat{x}_2 + r\hat{x}_3)/\epsilon)\,\mathrm{d}\hat{x}_1
\mathrm{d}\hat{x}_2\mathrm{d}\hat{x}_3}. \label{equ:B3}
\end{align}
\end{subequations} 
Denote the vertexes of $T$ by $a_i, (i=1,2,3,4)$. Define $\bar{\beta}
= \bar{\theta}\bar{\alpha}$ and $t_{ij} = a_j - a_i$. Below we give
the detailed formulations of local SAFE bilinear forms.
\begin{enumerate}
\item $k=0$: The local SAFE bilinear form \eqref{equ:aT2-k0} can be
implemented by 
\begin{equation} \label{equ:aT2-k0-Bernoulli}
\begin{aligned}
&\quad~ a_T(w_h, v_h) \\
&= \sum_{E = \overrightarrow{a_ia_j}}
\omega_E^T \bar{\alpha} 
\frac{1}{\dashint_E \exp(\bar{\beta}\cdot x/\bar{\alpha})} \\ 
&\qquad\qquad \big( \exp(\bar{\theta} \cdot a_j)w_h(a_j) 
     -\exp(\bar{\theta} \cdot a_i) w_h(a_i)\big)\delta_E(v_h) \\
&= \sum_{E = \overrightarrow{a_ia_j}} \omega_E^T
\big( B_1^{\bar{\alpha}}(\bar{\beta}\cdot t_{ji})w_h(a_j)
    - B_1^{\bar{\alpha}}(\bar{\beta}\cdot t_{ij})w_h(a_i)
\big) \delta_E(v_h).
\end{aligned}
\end{equation}
\item $k=1$: Note that, for any two faces $F =
\overrightarrow{a_ia_ja_k}$ and $F' = \overrightarrow{a_ia_ja_l}, (k
\neq l)$, the orientations must be different. Therefore, the local
SAFE bilinear form \eqref{equ:aT2-k1} can be implemented by 
\begin{equation} \label{equ:aT2-k1-Bernoulli}
\begin{aligned}
& \quad~ a_T(w_h, v_h) \\
& = \sum_{\substack{F = \overrightarrow{a_ia_ja_k} \\
F' = \overrightarrow{a_ia_ja_l}, k\neq l}} - \omega_{FF'}^T
\bar{\alpha} \frac{1}{\dashint_F \exp(\bar{\beta}\cdot x
    /\bar{\alpha})} \Big( 
l^1_{\overrightarrow{a_ia_j}}(w_h) \dashint_{\overrightarrow{a_ia_j}}
\exp(\bar{\theta}\cdot x) 
\\
& \qquad \quad 
+ l^1_{\overrightarrow{a_ja_k}}(w_h) \dashint_{\overrightarrow{a_ja_k}}
\exp(\bar{\theta}\cdot x) 
+ l^1_{\overrightarrow{a_ka_i}}(w_h) \dashint_{\overrightarrow{a_ka_i}}
\exp(\bar{\theta}\cdot x) 
\Big)  \\
&\qquad\quad \cdot \big( l^1_{\overrightarrow{a_ia_j}}(v_h) 
+ l^1_{\overrightarrow{a_ja_l}}(v_h)
+ l^1_{\overrightarrow{a_la_i}}(v_h) \big) \\
&= \sum_{\substack{F = \overrightarrow{a_ia_ja_k} \\
F' = \overrightarrow{a_ia_ja_l}, k\neq l}} - \omega_{FF'}^T
\big( 
B_2^{\bar{\alpha}}(\bar{\beta}\cdot t_{ij}, \bar{\beta}\cdot t_{ik}) 
l^1_{\overrightarrow{a_ia_j}}(w_h) \\
&\qquad \quad 
+ B_2^{\bar{\alpha}}(\bar{\beta}\cdot t_{jk}, \bar{\beta}\cdot t_{ji}) 
l^1_{\overrightarrow{a_ja_k}}(w_h)
+ B_2^{\bar{\alpha}}(\bar{\beta}\cdot t_{ki}, \bar{\beta}\cdot t_{kj}) 
l^1_{\overrightarrow{a_ka_i}}(w_h) \big) \\
& \qquad\quad \cdot \big( l^1_{\overrightarrow{a_ia_j}}(v_h) 
+ l^1_{\overrightarrow{a_ja_l}}(v_h)
+ l^1_{\overrightarrow{a_la_i}}(v_h) \big).
\end{aligned}
\end{equation}
Here, the degree of freedom $l^1_{\overrightarrow{a_ia_j}}(\cdot)$
corresponds to the orientation $\overrightarrow{a_ia_j}$.
\item $k=2$: The local SAFE bilinear form \eqref{equ:aT2-k2} can be
implemented by 
\begin{equation} \label{equ:aT2-k2-Bernoulli} 
\begin{aligned}
&\quad~ a_T(w_h, v_h) \\
& = \omega_T \big( B_3^{\bar{\alpha}}(\bar{\beta}\cdot t_{43},
    \bar{\beta}\cdot t_{42}, \bar{\beta}\cdot t_{41})l_{F_1}^2(w_h) \\
&\qquad + B_3^{\bar{\alpha}}(\bar{\beta}\cdot t_{14}, \bar{\beta}\cdot
  t_{13}, \bar{\beta}\cdot t_{12})l_{F_2}^2(w_h) \\
&\qquad + B_3^{\bar{\alpha}}(\bar{\beta}\cdot t_{21}, \bar{\beta}\cdot
  t_{24}, \bar{\beta}\cdot t_{23})l_{F_3}^2(w_h) \\
&\qquad + B_3^{\bar{\alpha}}(\bar{\beta}\cdot t_{32},
    \bar{\beta}\cdot t_{31}, \bar{\beta}\cdot t_{34})l_{F_4}^2(w_h) 
\big) \delta_T(v_h).
\end{aligned}
\end{equation}
Here, the degree of freedom $l_{F_i}^2(\cdot)$ corresponds to the unit
outer normal.
\end{enumerate}

\subsection{Limiting case} \label{subsec:limiting}
First we show that the Bernoulli functions \eqref{equ:Bernoulli}
remain viable when $\epsilon \to 0^+$. 
\begin{enumerate}
\item 1D Bernoulli function \eqref{equ:B1}: As $\epsilon \to 0^+$,
\begin{equation} \label{equ:B1-limit}
B_1^\epsilon(s) = \frac{s}{\exp(s/\epsilon) - 1} \to B_1^0(s) := 
\begin{cases}
-s & s \leq 0,\\ 
0 & s \geq 0.
\end{cases}
\end{equation}
\item 2D Bernoulli function \eqref{equ:B2}: As $\epsilon \to 0^+$,
\begin{equation} \label{equ:B2-limit}
\begin{aligned}
B_2^\epsilon(s,t) &=
\frac{t(t-s)(\exp(s/\epsilon)-1)}{2(s\exp(t/\epsilon) -
t\exp(s/\epsilon) + t - s)} \\ 
& \to B_2^0(s,t) := 
\begin{cases}
\frac{s-t}{2} & \max\{s,t\} = s \geq 0,\\
0 & \max\{s,t\} = t \geq 0, \\
-\frac{t}{2} & s\leq 0 \mbox{ and } t\leq 0. 
\end{cases}
\end{aligned}
\end{equation}
\item 3D Bernoulli function \eqref{equ:B3}: As $\epsilon \to 0^+$,
\begin{equation} \label{equ:B3-limit}
\begin{aligned}
B_3^\epsilon(s,t,r) &=
-\frac{r(s-r)(r-t)(s\exp(t/\epsilon) - t\epsilon(s/\epsilon) + t - s)}
{ \splitfrac{3\big( st(t-s)\exp(r/\epsilon) + sr(s-r)\exp(t/\epsilon)}
  {+ rt(r-t)\exp(s/\epsilon) + (t-s)(s-r)(r-t) \big)}} \\
& \to B_3^0(s,t,r) := 
\begin{cases}  
\frac{s-r}{3} & \max\{s,t,r\} = s \geq 0, \\
\frac{t-r}{3} & \max\{s,t,r\} = t \geq 0, \\
0 & \max\{s,t,r\} = r \geq 0, \\
-\frac{r}{3} & s \leq 0, t \leq 0, \mbox{ and } r \leq 0.
\end{cases}
\end{aligned}
\end{equation}
\end{enumerate}

In light of \eqref{equ:aT2-k0-Bernoulli}-\eqref{equ:B3-limit}, we
immediately see that the SAFE have limiting schemes when the diffusion
coefficient approaches to zero. The resulting schemes are special
upwind schemes according to limit of Bernoulli functions. 

\section*{Acknowledgments}
The authors would like to express their gratitude to Professor Ludmil
Zikatanov for his helpful discussions and suggestions.

%%%% Concluding %%%%
%%\section{Concluding Remarks} \label{sec:conclusions}
\bibliographystyle{siamplain}
\bibliography{SAFE_SINUM.bib} 
\end{document}